\documentclass[11pt, reqno]{amsart}
\usepackage{a4wide}
\usepackage{wrapfig}
\usepackage{amsmath}
\usepackage{amsthm}
\usepackage{mathrsfs}
\usepackage{color}
\usepackage{xcolor}
\usepackage{graphicx}
\usepackage{amssymb}
\usepackage{ dsfont }
\usepackage{url}
\usepackage{multicol}
\usepackage{tikz}
\usepackage{amsmath}
\usepackage{fancyhdr}
\usetikzlibrary{matrix}
\usetikzlibrary{arrows}
\usepackage{subfig}
\usepackage{hyperref}
\usepackage{listings}

\usepackage{esint}

\usepackage{enumitem}
\usepackage{varwidth}
\allowdisplaybreaks

\title[Orbital stability and instability of periodic wave solutions for the $\phi^4$-model]{Orbital stability and instability of periodic wave solutions for the $\phi^4$-model}

\author[J. M. Palacios]{Jos\'e Manuel Palacios}
\address{Institut Denis Poisson, Universit\'e de Tours, Universit\'e d'Orleans, CNRS, Parc Grandmont 37200, Tours, France}
\email{jose.palacios@lmpt.univ-tours.fr}

\newcommand{\be}{\begin{equation}}
\newcommand{\ee}{\end{equation}}
\newcommand{\bp}{\begin{proof}}
\newcommand{\ep}{\end{proof}}
\newcommand{\bel}{\begin{equation}\label}
\newcommand{\eeq}{\end{equation}}
\newcommand{\bea}{\begin{eqnarray}}
\newcommand{\eea}{\end{eqnarray}}
\newcommand{\bee}{\begin{eqnarray*}}
\newcommand{\eee}{\end{eqnarray*}}
\newcommand{\ben}{\begin{enumerate}}
\newcommand{\een}{\end{enumerate}}

\newcommand{\R}{\mathbb{R}}

\newcommand{\Z}{\mathbb{Z}}
\newcommand{\T}{\mathbb{T}}

\newtheorem{thm}{Theorem}[section]
\newtheorem{cor}[thm]{Corollary}
\newtheorem{lem}[thm]{Lemma}
\newtheorem{prop}[thm]{Proposition}

\theoremstyle{remark}
\newtheorem{rem}{Remark}[section]

\definecolor{codegreen}{rgb}{0,0.6,0}
\definecolor{codegray}{rgb}{0.5,0.5,0.5}
\definecolor{codepurple}{rgb}{0.58,0,0.82}
\definecolor{backcolour}{rgb}{0.95,0.95,0.92}

\lstdefinestyle{mystyle}{
	backgroundcolor=\color{backcolour},   
	commentstyle=\color{codegreen},
	keywordstyle=\color{magenta},
	numberstyle=\tiny\color{codegray},
	stringstyle=\color{codepurple},
	basicstyle=\footnotesize,
	breakatwhitespace=false,         
	breaklines=true,                 
	captionpos=b,                    
	keepspaces=true,                 
	numbers=left,                    
	numbersep=5pt,                  
	showspaces=false,                
	showstringspaces=false,
	showtabs=false,                  
	tabsize=2
}
\numberwithin{equation}{section}

\usepackage{pgfplots}
\pgfplotsset{compat=newest}



\theoremstyle{definition}

\numberwithin{ej}{section}

\begin{document}





\renewcommand{\sectionmark}[1]{\markright{\thesection.\ #1}}
\renewcommand{\headrulewidth}{0.5pt}
\renewcommand{\footrulewidth}{0.5pt}
\begin{abstract}
In this work we find explicit periodic wave solutions for the classical $\phi^4$-model, and study their corresponding orbital stability/instability in the energy space. In particular, for this model we find at least four different branches of spatially-periodic wave solutions, which can be written in terms of Jacobi elliptic functions. Two of these branches corresponds to \emph{superluminal} waves, a third-one corresponding to a \emph{sub-luminal} wave and the remaining one corresponding to a stationary complex-valued wave. In this work we prove the orbital instability of real-valued sub-luminal traveling waves. Furthermore, we prove that under some additional hypothesis, complex-valued stationary waves as well as the real-valued zero-speed sub-luminal wave are all stable. This latter case is related (in some sense) to the classical Kink solution.
\end{abstract}
\maketitle 
\section{Introduction}

\subsection{The model} This paper is concerned with the stability properties of traveling/standing wave solutions to the $1+1$ dimensional $\phi^4$ equation on the torus \begin{align}\label{phif}
\partial_{t}^2\phi-\partial_x^2\phi=\phi-\vert\phi\vert^2\phi, \quad t\in\R,\  x\in\T_L:=\R/L\Z.
\end{align}
Here, $\phi(t,x)$ denotes a scalar $L$-periodic function with values either in $\R$ or $\mathbb{C}$. This equation arises in Quantum Field Theory as a model for self-interactions of scalar fields (represented by $\phi$) and is one of the simplest examples where to apply Feynman diagram techniques to do perturbative analysis in quantum theory. Moreover, equation \eqref{phif} has also been derived as a simple continuum model of lightly doped polyacetylene \cite{Ri}. We refer the interested reader to \cite{MaPa,PeSc,Va} for some other physical motivations. 

\medskip

Equation \eqref{phif} can be understood as a particular case of the general family of nonlinear Klein-Gordon equations: \begin{align}\label{klein}
\partial_t^2\phi-\partial_x^2\phi=m\phi+f(\phi),
\end{align}
where $m\in\R$ and $f:\R\to\R$ denotes the nonlinearity. Many important nonlinear models can be recovered as particular cases of this latter equation, such as the $\phi^6$ and the sine-Gordon equations. Of course, different signs for the right-hand side will produce completely different dynamics of the solutions. However, under rather general assumptions, it is still possible to obtain some stability results for model \eqref{klein}. We refer the reader to \cite{De,KMM2} for a fairly general theory for small solution to equation \eqref{klein} and to \cite{LiSo,St} for studies of the long time asymptotics for some generalizations of equation \eqref{phif} with variable coefficients.

\medskip

On the other hand, since \eqref{phif} corresponds to a wave-like equation, it can be rewritten in the standard form as a first order system for $\vec{\phi}=(\phi_1,\phi_2)$ as \begin{align}\label{phif_2}
\begin{cases}
\partial_t\phi_1=\phi_2,
\\ \partial_t\phi_2=\partial_x^2\phi_1+\phi_1-\vert\phi_1\vert^2\phi_1.
\end{cases}
\end{align}
Moreover, from the Hamiltonian structure of the equation it follows that, at least formally, the energy of system \eqref{phif_2} is conserved along the trajectory, that is,
\begin{align}\label{energy} 
\mathcal{E}(\vec{\phi}(t))&:=\dfrac{1}{2}\int_0^L \big(\vert \phi_2\vert^2+\vert\phi_{1,x}\vert^2+\dfrac{1}{2}(1-\vert\phi_1\vert^2)^2\big)(t,x)dx=\mathcal{E}(\vec{\phi}_0).
\end{align}
In the real-valued case, the conservation of momentum shall also play a fundamental role for our current purposes, which is given by:
\begin{align}\label{momentum}
\mathcal{P}(\vec{\phi}(t)):=\int_0^L \phi_2(t,x)\phi_{1,x}(t,x)dx=\mathcal{P}(\vec{\phi}_0).
\end{align}
In the complex-valued case, we have the following additional conservation law which shall be key in our analysis:
\begin{align}\label{niE_nim}
\mathcal{F}\big(\vec{\phi}(t)\big):=\mathrm{Im}\int_0^L \overline{\phi}_1(t,x)\phi_2(t,x)dx =\mathcal{F}\big(\vec{\phi}_0\big),
\end{align}
where $\mathrm{Im}(\cdot)$ stands for the imaginary part of a complex number. We point out that from these conservation laws it follows that $H^1(\T_L)\times L^2(\T_L)$ defines the natural energy space associated to system \eqref{phif_2}.

\medskip

Additionally, equation \eqref{phif} is known for satisfying several symmetries. Among the most important ones we have the invariance under space and time translations. We point out that in the aperiodic setting there is an extra invariance, the so-called  Lorentz boost, that means, if $\vec{\phi}(t,x)$ is a solution to the equation, then so is \[
\vec{\varphi}(t,x):=\vec{\phi}\big(\gamma(t-\beta x),\gamma(x-\beta t)\big) \quad \hbox{where}\quad \gamma:= \sqrt{1-\beta^2} \quad \hbox{and}\quad \beta\in(-1,1). 
\]
However, notice that this transformation does not let the period fixed, and hence it is not and invariance of the equation in our current setting.

\medskip

On the other hand, two of the most important objects in nonlinear dynamics are traveling and standing wave solutions, particularly in the context of dispersive PDEs due to the so-called \emph{soliton conjecture}. The existence and (if the case) the corresponding orbital stability of such type of solutions has become a fundamental issue in the area. In this regard, we prove the existence of at least three different branches of traveling wave solutions to equation \eqref{phif} in the periodic setting, as well as one branch of standing wave solutions, which, up to the best of our knowledge, were not known for equation \eqref{phif} until now. These solutions are formally given by \begin{align}\label{rof_def}
\widetilde{\beta}_1\mathrm{dn}\big(\ell_1(x-ct);\kappa_1\big), \qquad &\qquad  \widetilde{\beta_2}\mathrm{cn}\big(\ell_2(x-ct);\kappa_2\big),
\\   \widetilde{\beta}_3\mathrm{sn}\big(\ell_3(x-ct);\kappa_3\big),\qquad &\qquad  \widetilde{\beta}_4e^{ict}\mathrm{sn}\big(\ell_4 x;\kappa_4\big),\label{rof_def_2}
\end{align}
where $\mathrm{dn}(\cdot;\cdot)$, $\mathrm{cn}(\cdot;\cdot)$ and $\mathrm{sn}(\cdot;\cdot)$ denotes the standard Jacobi elliptic functions (see Section \ref{preliminaries} for their  complete definition), and $\beta_i$, $\ell_i$ and $\kappa_i$ are positive parameters satisfying some specific relations where not all combinations are allow. See Propositions \ref{MT_DN_CURVE}, \ref{MT_CN_CURVE}, \ref{MT_SN_CURVE},  \ref{MT_SN_COMPLEX} for the specific definitions of each case respectively.

\medskip

One of the key points in our analysis is the use of classical results of Grillakis-Shatah-Strauss (see \cite{GSS,GSS2}) which set a general framework to study the orbital stability/instability for both traveling and standing wave solutions. These general results are based on the spectral information of the linearize Hamiltonian around these specific type of solutions. Thereby, it is worth to notice that, in the real-valued case, equation \eqref{phif_2} can be rewritten in the abstract Hamiltonian form as \[
\partial_t\vec\phi=\mathbf{J}\mathcal{E}'(\vec\phi) \quad \hbox{where} \quad \mathbf{J}:=\left(\begin{matrix}
0 & 1 \\ -1 & 0 
\end{matrix}\right),
\]
where $\mathcal{E}'$ denotes the Frechet derivative of the energy functional $\mathcal{E}$ (see Section \ref{os_c} for the analogous expression that we shall use for the complex valued case).

\medskip

Regarding the well-posedness of the equation, we recall that by applying classical Kato theory for quasilinear equations we obtain the local well-posedness in the energy space $H^1(\T_L)\times L^2(\T_L)$ of equation \eqref{phif} (see \cite{Kato}).  We refer the reader to  \cite{De,HaNa,Kl,Kl2} for several other local and global well-posedness results in one-dimensional and higher dimensional Klein-Gordon equations.

\medskip

About orbital stability of explicit solutions to equations \eqref{phif} and \eqref{klein}, there exists a vast literature regarding the aperiodic case. We refer the reader to \cite{HPW} for a classical and rather general result about the orbital stability of Kink solutions, and to \cite{KMM} for the first result regarding asymptotic stability of Kink solutions for equation \eqref{phif} (see also \cite{AlMuPa3} for a recent work in this direction). We also refer to \cite{Cu} for an study of the asymptotic stability properties of these solutions in dimension $3$. For the case of standing wave solutions $e^{ict}\phi(x)$ to Klein-Gordon equations, we refer to \cite{ShSt,ShSt2} for classical result in this setting. Nevertheless, for the periodic case there are only a few well-known results. We refer the reader to \cite{AnNa2,NaCa,NaPa} for the treatment of periodic solutions for a specific type of Klein-Gordon equations, when the sign of the right-hand side is the opposite one. Specifically, the first two of these works considers the stability problem of periodic solutions with $-\phi+\vert \phi\vert^4\phi $ as right-hand side in \eqref{phif}, while the third one considers $+\vert\phi\vert^2\phi$ and $-\phi+\vert\phi\vert^2\phi$ as right-hand sides. We point out that \eqref{phif} does not fit in any of these settings. We also refer the reader to \cite{AnNa} for the first work applying these type of techniques to obtain the main spectral information needed in our analysis. Regarding orbital stability of periodic wavetrains, we refer the reader to \cite{JoMaMiPl}. We remark that this latter result seems to be the first one (up to the best of our knowledge) for wavetrains in the periodic case (see also \cite{JoMaMiPl3}). On the other hand, we refer to \cite{DeMc,JoMaMiPl2} for orbital stability results in a particularly interesting Klein-Gordon setting (but different from the previous-ones), the so-called sine-Gordon equation. We point out that in these two works the authors also treat the case of \emph{superluminal waves}, case that we do not treat in this work. About stability of periodic traveling waves in Hamiltonian equations that are first-order in time, we refer to \cite{DeUp} for stability results for the nonlinear Schr\"odinger equation and to \cite{An,DeKa,DeNi} for the KdV and mKdV settings. Finally, we refer the reader to \cite{AlMuPa} for an stability study for more complex periodic structures that do not fit into the framework of Grillakis \emph{et al.} \cite{GSS}, such as spatiallty-periodic \emph{Breathers}. These are explicit solutions to the equation which behave as solitons but are also time-periodic. See also \cite{AlMuPa2,MP} for some stability results of aperiodic Breathers in the sine-Gordon equation.

\subsection{Main results}

In order to present our main results, let us first define what it means for a solution to be \emph{Orbitally Stable}. In the real-valued case, we say that a traveling wave solution $\vec{\varphi}_c$ is orbitally stable if for all $\varepsilon>0$ there exists $\delta>0$ small enough such that for every initial data $\vec{\phi}_0\in X$, with $X:=H^1(\T_L)\times L^2(\T_L)$, satisfying $
\Vert \vec{\phi}_0-\vec{\varphi}_c\Vert_{X}\leq\delta$, then
\[
\sup_{t\in\R}\inf_{\rho\in[0,L)}\Vert\vec{\phi}(t)-\vec{\varphi}_c(\cdot-\rho) \Vert_{X}<\varepsilon.
\]
Additionally, in the complex-valued case the latter condition has to be replaced by \[
\sup_{t\in\R}\inf_{\theta\in[0,2\pi)}\Vert\vec{\phi}(t)-e^{i\theta}\vec{\varphi}_c\Vert_{X}<\varepsilon.
\]
Otherwise, we say that $\vec{\varphi}_c$ is orbitally unstable. In particular, this latter is the case when the solution cease to exist in finite time. We point out that the differences in the previous definitions come from the differences in their orbits. While in the real valued case the orbit of $\vec\varphi_c$ is generated by the action of the translation group $T(s)u=u(\cdot-s)$, in the complex-valued case, since we are only dealing with standing wave solutions, their orbits are generated by the action of the gauge group $T(s)u=e^{is}u$ (see \eqref{rof_def} and \eqref{rof_def_2}). 

\medskip

It is worth to notice that, even when it is not explicitly said, we shall always assume that $L$ is the fundamental period of $\vec{\varphi}_c$. In particular, we are only considering perturbations with exactly the same period as our fundamental solution.

\medskip

Now, in order to avoid overly introducing new notation and definitions in this introductory section, we shall only formally state our main results.
\begin{thm}[Orbital Instability of snoidal waves: real-valued case]\label{MT3_SN}
The real-valued snoidal wave solution (see the first member of \eqref{rof_def_2}) is orbitally unstable in the energy space by the periodic flow of the $\phi^4$ equation. 
\end{thm}

Moreover, as a by-product of our analysis we also conclude the orbital instability for the standing case $c=0$. However, under some additional hypothesis we have the following result. 
\begin{thm}[Orbital Stability: stationary case]\label{MT_STATIONARY}
The real-valued stationary ($c=0$) snoidal wave solution (see the first member of \eqref{rof_def_2}) is orbitally stable by the periodic flow of the $\phi^4$ equation under $(\mathrm{odd},\mathrm{odd})$ perturbations in the energy space.
\end{thm}

\begin{rem}
We remark that the oddness character of the initial data is preserved by the periodic flow associated to equation \eqref{phif}. In other words, if $\vec{\phi}_0=(\phi_{0,1},\phi_{0,2})=(\mathrm{odd},\mathrm{odd})$, then so is the solution for all times. Then, noticing that the stationary snoidal wave $(S,0)$ correspond to an $(\mathrm{odd},\mathrm{odd})$ vector (see Section \ref{mt_r_s}), we obtain that, under the assumptions of the previous theorem, the solution shall always remain odd.  Here, and for the rest of this paper, when we refer to an \emph{odd} function, we mean that it is odd regarded as a function in the whole line.
\end{rem}

Finally we show that, in the complex-valued case, we have stability in the \emph{odd energy space}. 
\begin{thm}[Orbital Stability of snoidal waves: complex-valued case]\label{MT_COMPLEX}
The complex-valued snoidal wave solution (see the second member of \eqref{rof_def_2}) is orbitally stable  by the periodic $\phi^4$-flow under $(\mathrm{odd},\mathrm{odd})$ perturbations in the energy space.
\end{thm}

\begin{rem}
We point out that, since equation \eqref{phif} is also invariant under the maps: \[
u(t,x)\mapsto u(-t,x),\quad u(t,x)\mapsto -u(t,x) \quad \hbox{and} \quad u(t,x)\mapsto -u(-t,x),
\]
we also deduce Theorems \ref{MT3_SN}, \ref{MT_STATIONARY} and\footnote{Notice that, for the case of Theorem \ref{MT_COMPLEX}, all three of these symmetries preserves the oddness of the solution.} \ref{MT_COMPLEX} for both, snoidal and \emph{anti-snoidal} waves in both cases, moving to the left or right respectively.
\end{rem}

\subsection{Organization of this paper} This paper is organized as follow. In Section \ref{preliminaries} we introduce the main objects needed in our analysis and review their most important properties. In Sections \ref{curve_r} and \ref{curve_c} we deal with the existence of smooth curves of solitary waves in the real-valued and complex-valued case respectively. Then, in Section \ref{mt_r_t} we prove our main instability result for the real-valued case. In Section \ref{mt_r_s} we prove the Stability Theorem \ref{MT_STATIONARY}.  Finally, in Section \ref{os_c} we prove our stability theorem for the complex-valued case, that is, we prove Theorem \ref{MT_COMPLEX}.

\medskip

\section{Premiliminaries: Jacobi Elliptic Functions}\label{preliminaries}

We refer to \cite{An,BiFr} for more extensive and detailed properties of these functions. First of all, we introduce the complete elliptic integral of first kind:
\begin{align}\label{CEIFK}
K(k):=\int_0^1\dfrac{dx}{\sqrt{(1-x^2)(1-k^2x^2)}}=\int_0^{\pi/2}\dfrac{d\theta}{\sqrt{1-k^2\sin^2\theta}}, \quad k\in (0,1).
\end{align}
In addition, we also introduce the so-called complete elliptic integral of second kind:
\begin{align}\label{CEISK}
E(k):=\int_0^1\sqrt{\dfrac{1-k^2x^2}{1-x^2}}dx=\int_0^{\pi/2}\sqrt{1-k^2\sin^2\theta}d\theta, \quad k\in(0,1)
\end{align}
We remark that from now on, every time we write $E(k)$ or $E(\kappa)$ we shall be referring to this function (recall that we use $\mathcal{E}$ for the energy, see \eqref{energy}). These two elliptic integrals shall play a fundamental role in the following sections. Some basic properties of these functions are:
\begin{align}\label{EK_zero_one}
E(0)=\tfrac{\pi}{2}, \quad E(1)=1, \quad K(0)=\tfrac{\pi}{2}, \quad K(1)=+\infty
\end{align}
Moreover, these two function are strictly monotone. Specifically, for all $k\in(0,1)$ we have
\begin{align}\label{Eprime}
E'(k)<0,\quad E''(k)<0, \quad K'(k)>0, \quad K''(0)>0.
\end{align}
On the other hand, $E(k)<K(k)$ for all $k\in(0,1)$ and the maps $k\mapsto E+K$ and $k\mapsto EK$ are also strictly increasing. Additionally, $E$ and $K$ satisfies the following ODEs: \begin{align}\label{ek_derivative}
E'(k)=\dfrac{E(k)-K(k)}{k} \quad \hbox{ and } \quad K'(k)=\dfrac{E(k)-(1-k^2)K(k)}{k(1-k^2)}.
\end{align}
Now we turn our attention to Jacobi elliptic functions. First of all, consider the elliptic integral: \[
u(y,k):=\int_0^y\dfrac{dt}{\sqrt{(1-t^2)(1-k^2t^2)}}=\int_0^{\varphi}\dfrac{d\theta}{\sqrt{1-k^2\theta^2}},
\]
which defines a \textbf{strictly increasing} function with respect to $y$, and hence, it has a well-defined inverse given by the so-called snoidal function \begin{align}\label{sn_preliminaries}
y=\sin\varphi=\mathrm{sn}(u;k).
\end{align}
Once the snoidal function has been introduced, we are in position to define the so-called \emph{cnoidal} and \emph{dnoidal} elliptic functions, which are given by (respectively): \[
\mathrm{cn}(u;k)=\sqrt{1-y^2}=\sqrt{1-\mathrm{sn}^2(u;k)} \quad \hbox{and}\quad \mathrm{dn}(u;k)=\sqrt{1-k^2y^2}=\sqrt{1-k^2\mathrm{sn}^2(u;k)},
\]
requiring that $\mathrm{sn}(0;k)=0$, $\mathrm{cn}(0,k)=1$ and $\mathrm{dn}(0,k)=1$. All of these functions have parity properties, specifically, they are odd, even and even functions respectively. Moreover, they have well-known fundamental periods given by $4K(k)$, $4K(k)$ and $2K(k)$ respectively. Summarizing the most important properties of Jacobi elliptic functions we have \begin{align*}
\mathrm{sn}^2u+\mathrm{cn}^2u=1, \qquad & \qquad  \mathrm{dn}^2u+k^2\mathrm{sn}^2u=1, 
\\ -1\leq \mathrm{sn}u,\mathrm{cn}u\leq 1, \qquad & \qquad  \sqrt{1-k^2}\leq \mathrm{dn}u\leq 1. 
\end{align*}
Additionally, all of these functions satisfies some elliptic ODEs. Specifically, 
$\mathrm{sn}(x;k)$ satisfies:
\begin{align}\label{ode_sn}
\left(\dfrac{dy}{dx}\right)^2=(1-y^2)(1-\kappa^2y^2).
\end{align}
On the other hand, $\mathrm{cn}(x;k)$ and $\mathrm{dn}(x,k)$ satisfy the following equations (respectively): \[
\left(\dfrac{dy}{dx}\right)^2=(1-y^2)(1-k^2+k^2y^2) \quad \hbox{and} \quad \left(\dfrac{dy}{dx}\right)^2=(y^2-1)(1-k^2-y^2).
\]

\medskip

\section{Existence of a smooth curve of periodic traveling waves: Real-valued case}\label{curve_r}

\subsection{Superluminal  waves}
Our first goal is to establish the existence of smooth curves of periodic traveling wave solutions to equation \eqref{phif}. In fact, in this subsection we seek for solutions of the form $u(t,x)=\phi_c(x-ct)$. In addition, for the rest of this subsection we shall assume we are in the superluminal case, that is,  we shall consider $c^2>1$. Thus, plugging $\phi_c(x-ct)$ into the equation we obtain that if $u(t,x)$ is a traveling waves solution, then $\phi_c$ must satisfy: \begin{align}\label{solit_eq}
(c^2-1)\phi_c''=\phi_c-\phi_c^3.
\end{align}
Hence, by multiplying both sides of the equation by $\phi_c'$, and then integrating, we obtain that $\phi_c$ must to satisfy the following first-order differential equation in quadrature form: \begin{align}\label{int_eliptic}
(\phi_c')^2=\dfrac{1}{2\omega_{sl}}\left(2\phi_c^2-\phi_c^4+4A_{\phi_c}\right)=-\dfrac{1}{2\omega_{sl}}F_{\phi_c}(\phi_c),
\end{align}
where $A_{\phi_c}$ stands for an arbitrary integration constant, $\omega_{sl}:=c^2-1$ and the polynomial function $F_{\phi_c}$ is given by \[
F_{\phi_c}(z):=z^4-2z^2-4A_{\phi_c}.
\] 
It is worth to notice that in sharp contrast with the aperiodic setting, in this case we cannot assume that $A_{\phi_c}=0$ as usual. Moreover, regarding the factorization of the polynomial function $F_{\phi_c}$, notice that we only have two cases. In fact, either we have four real roots of the form $\pm\beta_1$, $\pm\beta_2$ or two real and two complex roots. That is, either we have \[
F_{\phi_c}(z)=(z^2-\beta_1^2)(z^2-\beta_2^2) \quad \hbox{or}\quad F_{\phi_c}(z)=(z^2+\beta_1^2)(z^2-\beta_2^2),
\]
where $\beta_1,\beta_2\in\R$ stands for the zeros of the polynomial $F_{\phi_c}$.

\medskip

\textbf{First case:} Our aim now is to find the periodic traveling wave solutions assuming that $F_{\phi_c}$ has exactly four real roots, that is, assuming that $$
F_{\phi_c}(z)=(z^2-\beta_1^2)(z-\beta_2^2),
$$
for some $\beta_1,\beta_2\in\R$. Notice that without loss of generality we can also assume $\beta_1>\beta_2>0$. Thus, if we seek for positive solutions\footnote{Notice that if $\phi_c$ is a solution, then so is $-\phi_c$, and hence, there is no loss of generality in this assumption.}, equation \eqref{int_eliptic} imposes the additional constraint $\phi_c(x)\in [\beta_2,\beta_1]$. Hence, if $F_{\phi_c}$ has four real roots it follows that \[
\beta_1^2+\beta_2^2=2 \quad \hbox{ and }\quad \beta_1^2\beta_2^2=-4A_{\varphi_c}.
\]
Now, for the sake of simplicity we define the auxiliary variables $\psi:=\beta_1^{-1}\phi_c$ and $\kappa^2=\beta_1^{-2}(\beta_1^2-\beta_2^2)$. Then, equation \eqref{int_eliptic} becomes \begin{align}\label{aux_eq_1}
(\psi')^2=-\dfrac{\beta_1^2}{2\omega_{sl}}(\psi^2-1)(\psi^2-1+\kappa^2)
\end{align}
Now we change variables once again. In fact, by considering the auxiliary variable given by the relation \[
\psi^2=:1-\kappa^2\sin^2\eta \quad \implies \quad 2\psi\psi'=-2\kappa^2\eta'\sin\eta \cos\eta,
\]
we obtain that equation \eqref{aux_eq_1} can be conveniently rewritten as \[
(\eta')^2=\dfrac{\beta_1^2}{2\omega_{sl}}(1-\kappa^2\sin^2\eta).
\]
Thus, by using Jacobi elliptic function we infer that $\sin \psi=\mathrm{sn}(\ell x;\kappa)$ where $\ell^2:=\tfrac{\beta_1^2}{2\omega_{sl}}$. Going back to the $\psi$ variable we obtain \[
\psi(x)=\sqrt{1-\kappa^2\mathrm{sn}^2(\ell x;\kappa)}=\mathrm{dn}(\ell x;\kappa).
\]
Therefore, the explicit solution to equation \eqref{solit_eq} is given by  $\phi_c(x)=\beta_1 \mathrm{dn}(\ell x;\kappa)$. On the other hand, notice that due to the relation $\beta_1^2+\beta_2^2=2$ and the fact that $\beta_1>\beta_2$ we deduce that $\beta_1\in(1,\sqrt{2})$. Finally, recall that the fundamental period of $\mathrm{dn}(\cdot;\kappa)$ is exaclty $2K(\kappa)$. Hence, we obtain that solution $u(t,x)$  has fundamental period (wavelength, denoted by $T_{\mathrm{dn}}$) given by \begin{align}\label{dn_period}
T_{\mathrm{dn}}:=\tfrac{2\sqrt{2\omega_{sl}}}{\beta_1}K(\kappa).
\end{align}
We point out that from the previous formula it follows that (a priori) the wavelength depends on the propagation speed $c$. However, as we shall see below, by making use of the parameter $\beta_1$ one can disengage $T_{\mathrm{dn}}$ from $c$, and hence, for $T_{\mathrm{dn}}$ fixed, there shall exists a whole family of traveling waves solutions with different speeds and the same period (see Proposition \ref{MT_DN_CURVE}). 

\begin{rem}[Range of the wavelength] We have the following scenarios:

\smallskip

\textbf{Case $\beta_1\to 1^+$:} Noticing that $\beta_2\to 1^-$ as $\beta_1\to 1^+$, it immediately follows from the relation \[\kappa^2=\tfrac{1}{\beta_1^{2}}(\beta_1^2-\beta_2^2)\] that $\kappa(\beta_1) \to0^+$ as $\beta_1\to 1^+$. On the other hand, from \eqref{EK_zero_one} we know that $K(0)=\tfrac{\pi}{2}$, and hence we obtain that  $T_{\mathrm{dn}}\to \pi\sqrt{2\omega}$.

\smallskip

\textbf{Case $\beta_1\to\sqrt{2}^-$:} In this case, by using the relation between $\kappa$ and $(\beta_1,\beta_2)$ again, we deduce that $\kappa(\beta_1)\to1^{-}$ as $\beta_1\to\sqrt{2}^{-}$. Then, recalling that from \eqref{EK_zero_one} we know that $K(1)=+\infty$, we infer that  $T_{\mathrm{dn}}\to+\infty$. 
\end{rem}

\begin{rem}\label{dn_almost_proof}
It is worth to notice that for any given period $L>0$ and any given speed \begin{align}\label{dn_c_interval} \vert c\vert \in \left(1,\sqrt{1+\tfrac{L^2}{2\pi^2}}\right),
\end{align} there exists a unique pair $(\beta_1,-A_{\varphi_c})\in (1,\sqrt{2})\times (0,\tfrac{1}{4})$ such that the corresponding $\mathrm{dn}(\cdot,\cdot)$ wave solution found above satisfies $
T_{\mathrm{dn}}=L$. In fact, it is enough to notice that, by writing $T_{\mathrm{dn}}$ as a function of $\beta_1$, we have\footnote{We shall rigorously prove this inequality in the proof of Proposition \ref{MT_DN_CURVE} below.} that $\tfrac{d}{d\beta_1}T_{\mathrm{dn}}>0$ for all $\beta_1\in(1,\sqrt{2})$, together with the fact that $T_{\mathrm{dn}}((1,\sqrt{2}))=(\pi\sqrt{2\omega_{sl}},+\infty)$. Notice also that under condition \eqref{dn_c_interval}, we have the bound $\pi\sqrt{2\omega_{sl}}<L$. Then, we conclude by applying the Implicit Function Theorem. We point out that once we choose $\beta_1$, the arbitrary constant $A_{\phi_c}$ is fixed by the relation $\beta_1^2(2-\beta_1^2)=-4A_{\phi_c}$.  
\end{rem}

\smallskip

Gathering all the above information we are in position to conclude the following proposition.
\begin{prop}[Smooth curve of dnoidal waves]\label{MT_DN_CURVE}
Let $L>0$ be arbitrary but fixed. Then, for any speed $c$ satisfying $$\vert c\vert \in\left(1,\sqrt{1+\tfrac{L^2}{2\pi^2}}\right),$$
there exists a unique $\beta_1=\beta_1(c)\in(1,\sqrt{2})$ such that the dnoidal wave solution to equation \eqref{phif} given by \[
u(t,x)=\beta_1\mathrm{dn}\big(\ell(x-ct);\kappa\big), \quad \hbox{where}\quad \ell^2=\tfrac{\beta_1^2}{2\omega_{sl}} \quad \hbox{and}\quad \kappa^2=\tfrac{2(\beta_1^2-1)}{\beta_1^2},
\]
has fundamental period $T_{\mathrm{dn}}=L$ and satisfies equation \eqref{solit_eq}, where $\omega_{sl}=c^2-1$. Moreover, the map $c\mapsto u(0,x)\in H^1(\T_L)$ is smooth.
\end{prop}

\begin{proof}
In fact, as we already discussed, it only remains to prove inequality $\tfrac{d}{d\beta_1}T_{\mathrm{dn}}>0$. We anticipate that, since the computations become cumbersome due to the appearance of many terms, we shall successively seek to reduce the problem to simpler equivalent problems. Indeed, by explicitly differentiating relation \eqref{dn_period} we obtain:
\begin{align}\label{db_T_dn}
\dfrac{d}{d\beta_1}T_{\mathrm{dn}}=-\dfrac{2\sqrt{2\omega_{sl}}}{\beta_1^2}K(\kappa)+\dfrac{2\sqrt{2\omega_{sl}}}{\beta_1}\cdot\dfrac{dK(\kappa)}{d\kappa}\cdot\dfrac{d\kappa(\beta_1)}{d\beta_1}.
\end{align}
On the other hand, by using the explicit form of $K$ (see \eqref{CEIFK}) we obtain that 
\begin{align}\label{dK_dk_dn}
\dfrac{dK}{d\kappa}=\dfrac{E(\kappa)-(1-\kappa^2)K(\kappa)}{\kappa(1-\kappa^2)}\quad \hbox{and}\quad 
\dfrac{d\kappa}{d\beta_1}=\dfrac{2}{\kappa \beta_1^3},
\end{align}
where $E$ denotes the complete elliptic integral of second kind (see \eqref{CEISK}). Then, by plugging \eqref{dK_dk_dn} into \eqref{db_T_dn} and after some trivial re-arrengements, we infer that proving inequality $\tfrac{d}{d\beta_1}T_{\mathrm{dn}}>0$ is equivalent to prove that \begin{align}\label{dtdb_equiv}
\dfrac{E(\kappa)-(1-\kappa^2)K(\kappa)}{\kappa^2(1-\kappa^2)}>\dfrac{\beta_1^2}{2}K(\kappa).
\end{align}
Hence, we turn our attention to prove inequality \eqref{dtdb_equiv}. In fact, first of all we recall that $1-\kappa^2>0$. Then, notice that the latter inequality can be conveniently rewritten as: \begin{align}\label{proof_dn_ineq_E_K}
E(\kappa)>\kappa^2(1-\kappa^2)\Big(\dfrac{\beta_1^2}{2}+\dfrac{1}{\kappa^2}\Big)K(\kappa).
\end{align}
Moreover, notice that the factor multiplying $K(\kappa)$ on the latter inequality can be simplified: \[
\kappa^2(1-\kappa^2)\Big(\dfrac{\beta_1^2}{2}+\dfrac{1}{\kappa^2}\Big)=2-\beta_1^2.
\]
Hence, from now on we shall work with the function $
F(\beta_1):=E(\kappa)+(\beta_1^2-2)K(\kappa)$.
Thus, in order to prove that inequality \eqref{proof_dn_ineq_E_K} holds, it is enough to show that $F(\beta_1)>0$ for all $\beta_1\in(1,\sqrt{2})$. Indeed, first of all notice that \[
\lim_{\beta_1\to1^+}F(\beta_1)=0.
\]
Then, the proof reduces to prove that $F'(\beta_1)>0$ for all $\beta_1\in(1,\sqrt{2})$. In fact, notice that on the one-hand by direct computations we have
\begin{align}\label{prime_F_proof}
F'(\beta_1)=2\beta_1K+\dfrac{2}{\kappa\beta_1^3}\left(E'(\kappa)+(\beta_1^2-2)K'(\kappa)\right).
\end{align}
On the other hand, we recall that the complete elliptic integral $E$ satisfies (see \eqref{CEISK}): 
\[
E'(k)=\dfrac{E(k)-K(k)}{k}
\]
Thus, gathering both expressions for $K'$ and $E'$ respectively and plugging them into \eqref{prime_F_proof} we obtain:
\begin{align*}
F'(\beta_1)&=\left(2\beta_1-\dfrac{2}{\kappa^2\beta_1^3}-\dfrac{2(\beta_1^2-2)}{\kappa^2\beta_1^3}\right)K(\kappa)+\dfrac{2}{\kappa\beta_1^3}\left(\dfrac{1}{\kappa}+\dfrac{\beta_1^2-2}{\kappa(1-\kappa^2)}\right)E(\kappa)
\\ & =:\mathrm{A}_1(\beta_1)K(\kappa)+\mathrm{A}_2(\beta_1)E(\kappa).
\end{align*}
Now, notice that, by using the explicit form of $\kappa(\beta_1)$, functions $\mathrm{A}_1$ and $\mathrm{A}_2$ can be significantly simplified as:
\[
\mathrm{A}_1(\beta_1)=2\beta_1-\dfrac{1}{\beta_1}>1 \quad \hbox{and} \quad \mathrm{A}_2(\beta_1)=-\dfrac{1}{\beta_1}.
\]
Therefore, recalling that $K(k)>E(k)$ for all $k\in(0,1)$ we conclude $F'(\beta_1)>0$ for all $\beta_1\in(1,\sqrt{2})$ which is the positivity we were looking for, and hence inequality \eqref{proof_dn_ineq_E_K} holds. Then, the proof follows from the Implicit Function Theorem, what finish the proof.
\end{proof}

\begin{figure}[h!]
   \centering
   \includegraphics[scale=0.45]{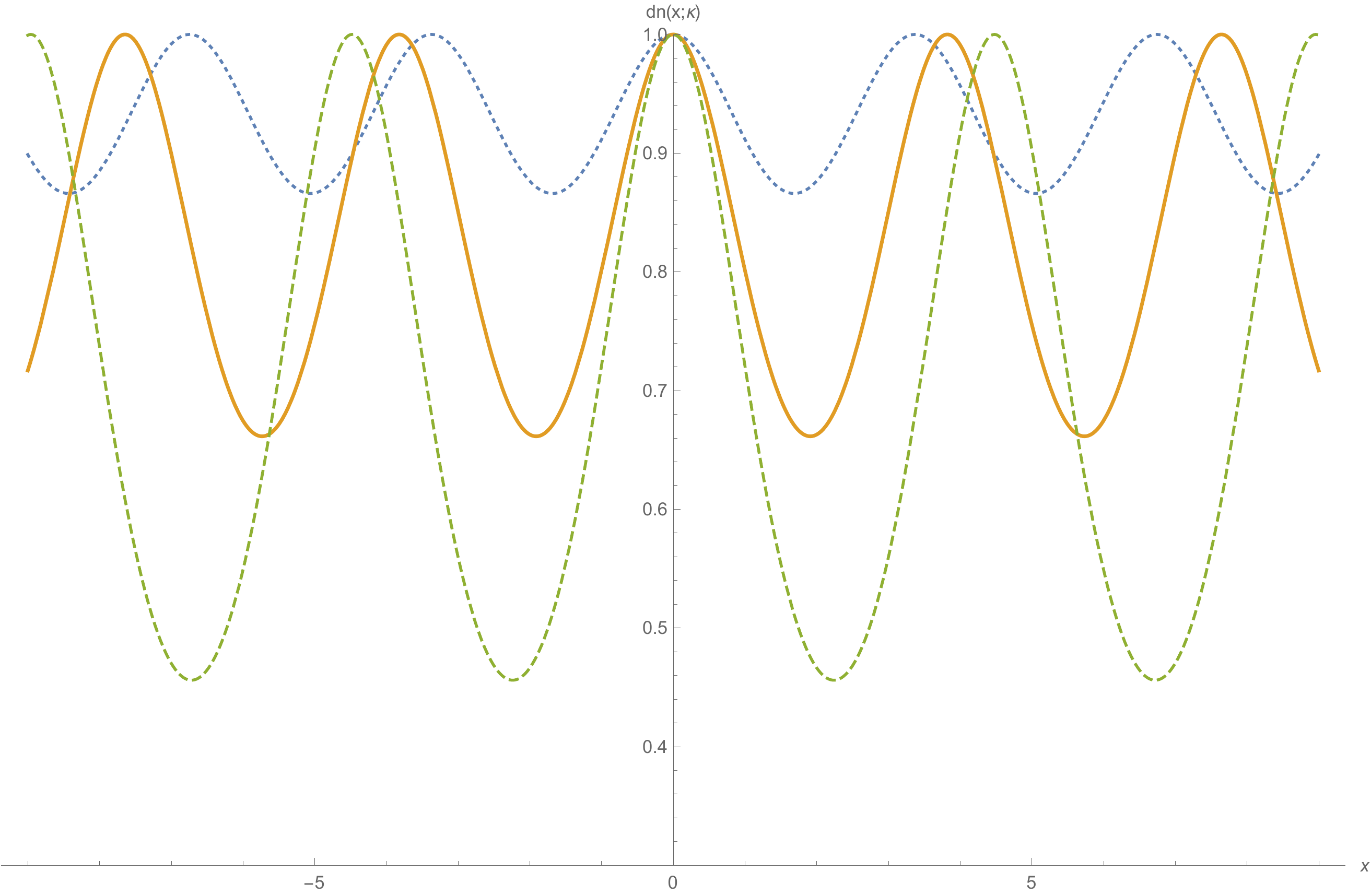}
   \caption{Dnoidal wave for different values of $\kappa\in(0,1)$. The dotted blue line represents $\mathrm{dn}(x;\kappa)$ for $\kappa=0.5$. On the thick orange line we have $\mathrm{dn}(x;\kappa)$ for $\kappa=0.75$. Finally, the dashed green line represents $\mathrm{dn}(x;\kappa)$ for $\kappa=0.9$. Notice that all of them are spatially-even and never vanish.}\label{Fig:1}
\end{figure}

\textbf{Second case:} Our goal now is to study the existence of periodic solitary waves in the case when $F_{\phi_c}$ can be factorized as $$
F_{\phi_c}(z)=(z^2+\beta_1^2)(z^2-\beta_2^2).
$$
Again, without loss of generality we can assume that $\beta_2>0$. Moreover, notice that due to equation \eqref{int_eliptic} and the current form of $F_{\phi_c}$, we have $\phi_c(x)\in[-\beta_2,\beta_2]$. On the other hand, from \eqref{int_eliptic} we also obtain that $\beta_1$, $\beta_2$ satisfy the equations \[
\beta_2^2-\beta_1^2=2 \quad \hbox{and}\quad \beta_1^2\beta_2^2=4A_{\varphi_c}.
\]
We proceed similarly as before. In fact, in order to write the equation in a standard (easily solvable) form, we define the new variables $\psi$ and $\kappa$ given by: \[
\psi(x):=\tfrac{1}{\beta_2}\phi_c(x) \quad \hbox{and}\quad \kappa^2:=\tfrac{\beta_2^2}{\beta_1^2+\beta_2^2}.
\]
Thus, by plugging these new variables into equation \eqref{int_eliptic} we get \begin{align}\label{aux_eq_one_cn}
(\psi')^2=-\dfrac{\beta_2^2}{2\omega_{sl}}(\psi^2+\tfrac{\beta_1^2}{\beta_2^2})(\psi^2-1).
\end{align}
Now, in order to write the equation in a standard form, we change variables once again by defining $\eta$ given by the relation $\psi^2=:\cos^2\eta$. Hence, by differentiating the previous relation we obtain that $2\psi\psi'=-2\eta'\sin\eta\cos\eta$ , and therefore, replacing into equation \eqref{aux_eq_one_cn} we obtain \[
(\eta')^2=\dfrac{\beta_1^2+\beta_2^2}{2\omega_{sl}}(1-\kappa^2\sin^2\eta),
\]
where we have used the fact that $\cos^2x=1-\sin^2x$. Then, in the same fashion as before, by using Jacobi elliptic functions we deduce that \[
\psi^2(x)=1-\mathrm{sn}^2(\ell x;\kappa)=\mathrm{cn}^2(\ell x,\kappa), \quad\hbox{where}\quad \ell^2:=\tfrac{1}{2\omega_{sl}}(\beta_1^2+\beta_2^2).
\]
Hence, going back to our original variable we obtain $\phi_c(x)$ is given by $\phi_c=\beta_2\mathrm{cn}(\ell x;\kappa)$. Thus, recalling that  $\beta_2^2-\beta_1^2=2$, and since we have assumed $\beta_2$ positive, we infer that $\beta_2>\sqrt{2}$. Therefore, we have found a second periodic wave solution to equation \eqref{phif}
given by \begin{align}\label{cn_u_kappa}
u(t,x):=\beta_2\mathrm{cn}\big(\ell(x-ct);\kappa\big)\quad \hbox{where}\quad  \kappa^2=\dfrac{\beta_2^2}{2\beta_2^2-2}\quad \hbox{and}\quad \ell^2=\dfrac{\beta_2^2-1}{\omega_{sl}}.
\end{align}
Finally, we recall that $\mathrm{cn}(\cdot,\kappa)$ has fundamental period $4K(\kappa)$ (see Section \eqref{preliminaries}), and hence, $u(t,x)$ has fundamental period (denoted by $T_{\mathrm{cn}}$ from now on): \begin{align}\label{period_cn}
T_{\mathrm{cn}}=\tfrac{4\sqrt{\omega_{sl}}}{\sqrt{\beta_2^2-1}}K(\kappa),
\end{align}
and just as before, we conclude that a priori its period (wavelength) depends on its speed $c$. However, we shall prove again that by taking advantage of $\beta_2$, it is possible to disengage $T_{\mathrm{cn}}$ from $c$, and hence, for $T_{\mathrm{cn}}$ fixed, there exists a whole family of traveling waves solutions with different speeds and the same period (see Proposition \ref{MT_CN_CURVE}).

\begin{rem}[Range of the wavelength, second case] Notice that in this case we only have one interesting scenario: In fact, from \eqref{cn_u_kappa} we immediately see that $\kappa(\beta_2) \to 1$ as $\beta_2\to \sqrt{2}^+$. Thus, recalling that $K(1)=+\infty$, we deduce that in this case $T_{\mathrm{cn}}\to +\infty$. The case when we let $\beta_2$ tends to $+\infty$ we obtain a wave which oscillates faster and faster,  with a greater amplitude each time (so that in this case $T_{\mathrm{cn}}\to0$).
\end{rem}

\begin{rem}
It is worth to notice that for any given period $L>0$ and any given speed \[
\vert c\vert \in(1,+\infty),
\] there exists a unique pair $(\beta_2,A_{\phi_c})\in (\sqrt{2},+\infty)\times (0,+\infty)$ such that the corresponding $\mathrm{cn}(\cdot,\cdot)$ wave solution found in \eqref{cn_u_kappa} satisfies: \[
T_{\mathrm{cn}}=L.
\]
In fact, it is enough to notice that in this case\footnote{We shall rigorously prove this inequality in the proof of Proposition \ref{MT_CN_CURVE} below.} we have $\tfrac{d}{d\beta_2}T_{\mathrm{cn}}<0$ for all $\beta_2\in(\sqrt{2},+\infty)$, together with the fact that $T_{\mathrm{cn}}((\sqrt{2},+\infty))=(0,+\infty)$. Then, we conclude by applying the Implicit Function Theorem. 
\end{rem}

Gathering all the above information we are in position to conclude the following Proposition.

\begin{prop}[Smooth curve of cnoidal waves]\label{MT_CN_CURVE}
Let $L>0$ be arbitrary but fixed. Then, for any speed $c$ satisfying $\vert c\vert \in\left(1,+\infty\right)$, there exists a unique $\beta_2\in(\sqrt{2},+\infty)$ such that the dnoidal wave solution to equation \eqref{phif} given by  \[
u(t,x)=\beta_2\mathrm{cn}\big(\ell(x-ct);\kappa\big), \quad \hbox{where}\quad \ell^2=\tfrac{\beta_2^2-1}{\omega_{sl}} \quad \hbox{and}\quad \kappa^2=\tfrac{\beta_2^2}{2\beta_2^2-2},
\]
has fundamental period $T_{\mathrm{cn}}=L$ and satisfies equation \eqref{solit_eq}, where $\omega_{sl}=c^2-1$. Moreover, the map $c\mapsto u(0,x)\in H^1(\T_L)$ is smooth.
\end{prop}

\begin{proof}
In fact, just as in the previous case, it only remains to prove inequality $\tfrac{d}{d\beta_2}T_{\mathrm{cn}}<0$. Indeed, first of all notice that we can rewrite $(\beta_2^2-1)^{-1/2}=\sqrt{2}\beta_2^{-1}\kappa$. Hence, \begin{align}
\dfrac{d}{d\beta_2}T_{\mathrm{cn}}=-\dfrac{4\sqrt{2\omega_{sl}}\kappa}{\beta_2^2}K+\dfrac{4\sqrt{2\omega_{sl}}\kappa'}{\beta_2}K+\dfrac{4\sqrt{2\omega_{sl}}\kappa'\kappa}{\beta_2}\dfrac{dK}{d\kappa}.
\end{align}
Then, the problem is reduced to prove $
-\tfrac{\kappa}{\beta_2}K(\kappa)+\kappa'K(\kappa)+\kappa'\kappa K'(\kappa)<0$. Now we recall again that due to the explicit form of $\kappa$ and the complete elliptic integral $K$, we have the following formulas (see \eqref{ek_derivative}):
\[
\dfrac{dK}{d\kappa}=\dfrac{E(\kappa)-(1-\kappa^2) K(\kappa)}{\kappa(1-\kappa^2)}\quad\hbox{and}\quad \dfrac{d\kappa}{d\beta_2}=-\dfrac{1}{\sqrt{2}(\beta_2^2-1)^{3/2}}.
\]
Thus, gathering both identities we obtain that \[
-\dfrac{\kappa}{\beta_2}K+\kappa'K+\kappa'\kappa K'=-\dfrac{\kappa}{\beta_2}K-\dfrac{1}{\sqrt{2}(\beta_2^2-1)^{3/2}(1-\kappa^2)}E<0,
\]
what concludes the proof.
\end{proof}

\begin{figure}[h!]
   \centering
   \includegraphics[scale=0.45]{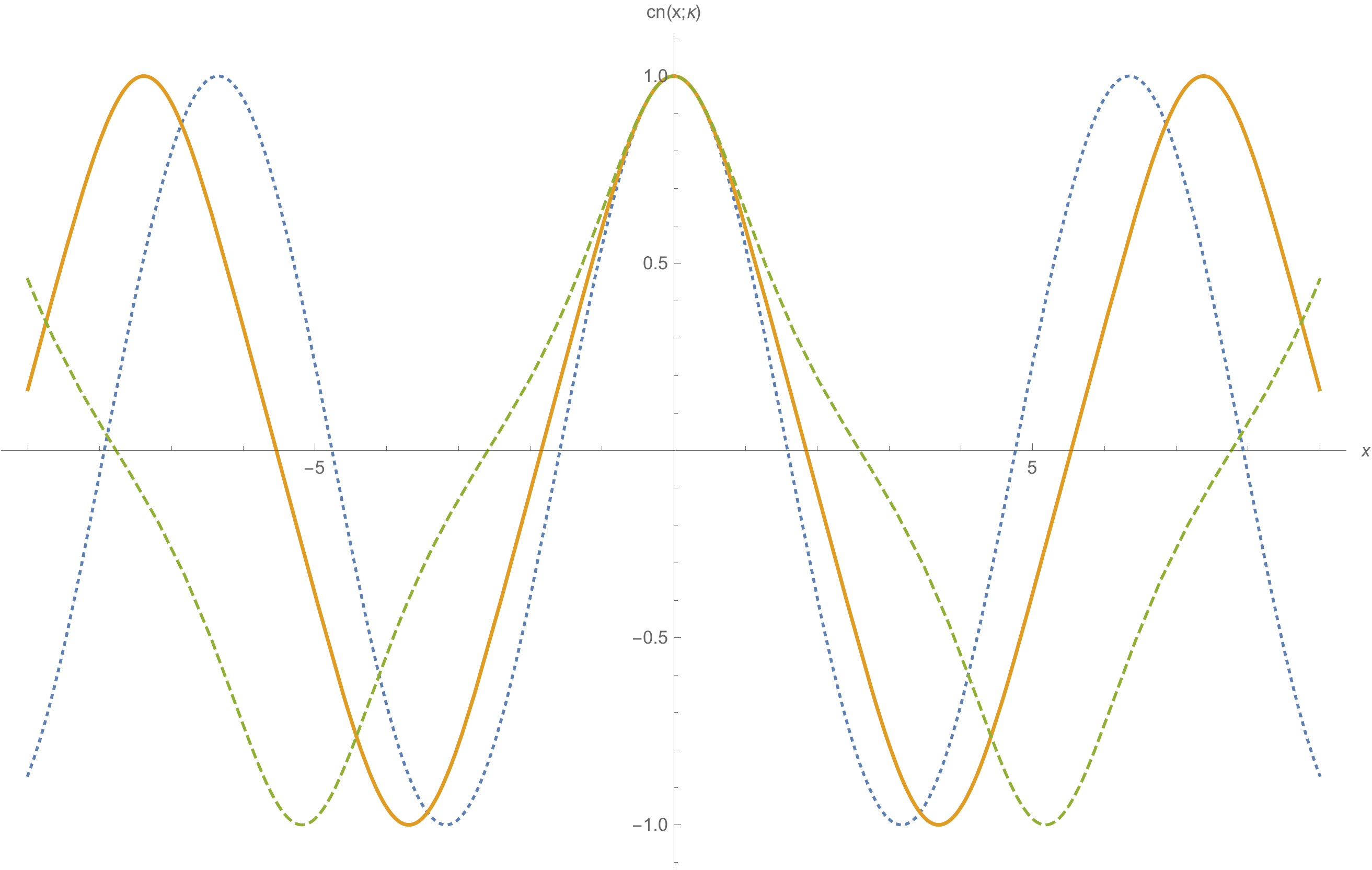}
   \caption{Cnoidal wave curve for different values of $\kappa$. The dotted blue line represents $\mathrm{cn}(x;\kappa)$ for $\kappa=0.2$. On the thick orange line we have $\mathrm{dn}(x;\kappa)$ for $\kappa=0.7$. Finally, the dashed green line represents $\mathrm{dn}(x;\kappa)$ for $\kappa=0.95$. Notice that all of them are spatially-even and have two zeros in each period.}\label{Fig:1}
\end{figure}

\subsection{Sub-luminal waves}

In this subsection we seek for \emph{sub-luminal} waves, that is, solutions to equation \eqref{phif} of the form $u(t,x)=\phi_c(x-ct)$ where $c$ is assumed to takes values $c\in(-1,1)$. Then, in a similar fashion as in the previous subsection, by plugging these type of functions into the equation we obtain that if $u(t,x)$ is a solution to equation \eqref{phif}, then $\phi_c$ must satisfy: \begin{align}\label{subluminal_der}
-(1-c^2)\phi_c''=\phi_c-\phi_c^3.
\end{align}
Hence, by multiplying both sides of the equation by $\phi_c'$ and integrating, we obtain that $\phi_c$ must to satisfies the first-order differential equation in quadrature form: \begin{align}\label{subluminal_int}
(\phi_c')^2=\dfrac{1}{2\omega_{\mathrm{sb}}}\left(\phi_c^4-2\phi_c^2-4A_{\phi_c}\right)=\dfrac{1}{2\omega_{\mathrm{sb}}}F_{\phi_c}(\phi_c),
\end{align}
where $A_{\phi_c}$ stands for an arbitrary integration constant, $w_{\mathrm{sb}}:=1-c^2$ and the polynomial function $F_{\phi_c}$ is given by \[
F_{\phi_c}(z):=z^4-2z^2-4A_{\phi_c}.
\] 
Thus, in contrast with the previous cases, now  we assume that $F_{\phi_c}$ has exactly four real roots, that is, we assume that $F_{\phi_c}$ can be factorize as: \begin{align}\label{roots_sb_sn}
F_{\phi_c}(z)=(z^2-\beta_1^2)(z^2-\beta_2^2) \quad \hbox{where} \quad \beta_1^2+\beta_2^2=2 \quad \hbox{ and }\quad \beta_1^2\beta_2^2=-4A_{\varphi_c}.
\end{align}
Without loss of generality we can also assume that $\beta_1>\beta_2>0$. Now, we seek for sign changing solutions, and hence equation \eqref{subluminal_int} imposes the additional constraint $\phi_c(x)\in [-\beta_2,\beta_2]$. Thus, in the same fashion as before, in order to write the equation in a more standard (easily solvable) form, we define the auxiliary variables $\psi:=\beta_2^{-1}\phi_c$ and $\kappa^2=\beta_2^{-2}(\beta_1^2-\beta_2^2)$. Then, equation \eqref{subluminal_int} becomes \begin{align}\label{subl_aux}
(\psi')^2=\dfrac{\beta_1^2}{2\omega_{\mathrm{sb}}}(1-\psi^2)\left(1-\tfrac{\beta_2^2}{\beta_1^2}\psi^2\right)
\end{align}
Notice that the latter ODE is in the form of \eqref{ode_sn}. Hence, by using the snoidal wave function defined in \eqref{sn_preliminaries} we obtain the explicit solution to equation \eqref{phif}:\begin{align}\label{subl_proof_def_sol}
u(t,x):=\beta_2\mathrm{sn}\big(\ell(x-ct)\big) \quad \hbox{where} \quad \ell:=\dfrac{\beta_1}{\sqrt{2\omega_{\mathrm{sb}}}},\quad \kappa:=\dfrac{\beta_2}{\beta_1} \quad \hbox{and}\quad \beta_2^2:=2-\beta_1^2.
\end{align}
Notice that from \eqref{roots_sb_sn} and the fact that $\beta_1>\beta_2$ we infer that $\beta_1\in(1,\sqrt{2})$. Finally, we recall that $\mathrm{sn}(\cdot,\kappa)$ has period $4K(\kappa)$, and hence, $u(t,x)$ has fundamental period (wavelength, denoted by $T_{\mathrm{sb}}$) given by \begin{align}\label{sb_period}
T_{\mathrm{sb}}:=\tfrac{4\sqrt{2 \omega_{\mathrm{sb}}}}{\beta_1}K(\kappa).
\end{align}

\begin{rem}[Range of the wavelength] In this case we have the following scenarios:

\smallskip

\textbf{Case $\beta_1\to 1^+$:} By taking limit in \eqref{subl_proof_def_sol} it follows that $\kappa(\beta_1) \to1^-$ as $\beta_1$ tends to $1^+$. Then, recalling that $K(1)=+\infty$, we obtain that  $T_{\mathrm{sb}}\to +\infty$ as $\beta_1$ tends to $1^+$.

\smallskip

\textbf{Case $\beta_1\to\sqrt{2}^-$:} In this case, by using formula \eqref{subl_proof_def_sol} again, we deduce that $\kappa(\beta_1)\to0^{+}$. Then, recalling that $K(0)=\tfrac{\pi}{2}$, we obtain $T_{\mathrm{sb}}\to2\pi\sqrt{\omega_{\mathrm{sb}}}$. 
\end{rem} 

Gathering all the above information we are in position to conclude the following Proposition. 
\begin{prop}[Smooth curve of snoidal waves]\label{MT_SN_CURVE}
Let $L>0$ be arbitrary but fixed. Then, for any speed $c$ satisfying \begin{align}\label{csb_condition}
\vert c\vert  \in \left(c_{\mathrm{sb}},1\right) \quad \hbox{where} \quad c_{\mathrm{sb}}^2=\max\left\{0,1-\tfrac{L^2}{4\pi^2}\right\},
\end{align}
there exists a unique $\beta_1=\beta_1(c)\in(1,\sqrt{2})$ such that the dnoidal wave solution to equation \eqref{phif} given by \begin{align}\label{sn_sol_def}
u(t,x):=\beta_2\mathrm{sn}\big(\ell(x-ct)\big) \quad \hbox{where} \quad \ell:=\dfrac{\beta_1}{\sqrt{2\omega_{\mathrm{sb}}}},\quad \kappa:=\dfrac{\beta_2}{\beta_1} \quad \hbox{and}\quad \beta_2^2:=2-\beta_1^2,
\end{align}
has fundamental period $T_{\mathrm{sb}}=L$ and satisfies equation \eqref{subluminal_der}, where $\omega_{sb}=1-c^2$. Moreover, the map $c\mapsto u(0,x)\in H^1(\T_L)$ is smooth.
\end{prop}

\begin{proof}
See the Appendix \ref{app_dn_sb_case}.
\end{proof}

\begin{rem}[Real-valued periodic standing wave solutions]
It is worth to notice that in the previous proposition, if $L>2\pi$ we may consider letting $c\to0^+$. In fact, in this case we have found a real-valued periodic standing wave solution to equation \eqref{phif}. Moreover, notice that once setting $c=0$, by letting $\beta_1\to 1^+$ we have that: \[
\ell \to \tfrac{1}{\sqrt{2}} \quad \hbox{ and } \quad K(\kappa)\to+\infty.
\]
Additionally, due to the fact that $\mathrm{sn}(x,1)=\tanh(x)$, in this case we formally recover the standard Kink solution \begin{align}\label{standing_kink}
H(x):=\tanh\big(\tfrac{x}{\sqrt{2}}\big).
\end{align}
We refer to \cite{HPW,KMM} for some studies regarding the orbital and asymptotic stability properties of the Kink solution. Of course, since we are setting $c=0$, for each period $L>2\pi$, there exist \textbf{only one} of these standing waves. 
\end{rem}

\begin{figure}[h!]
   \centering
   \includegraphics[scale=0.45]{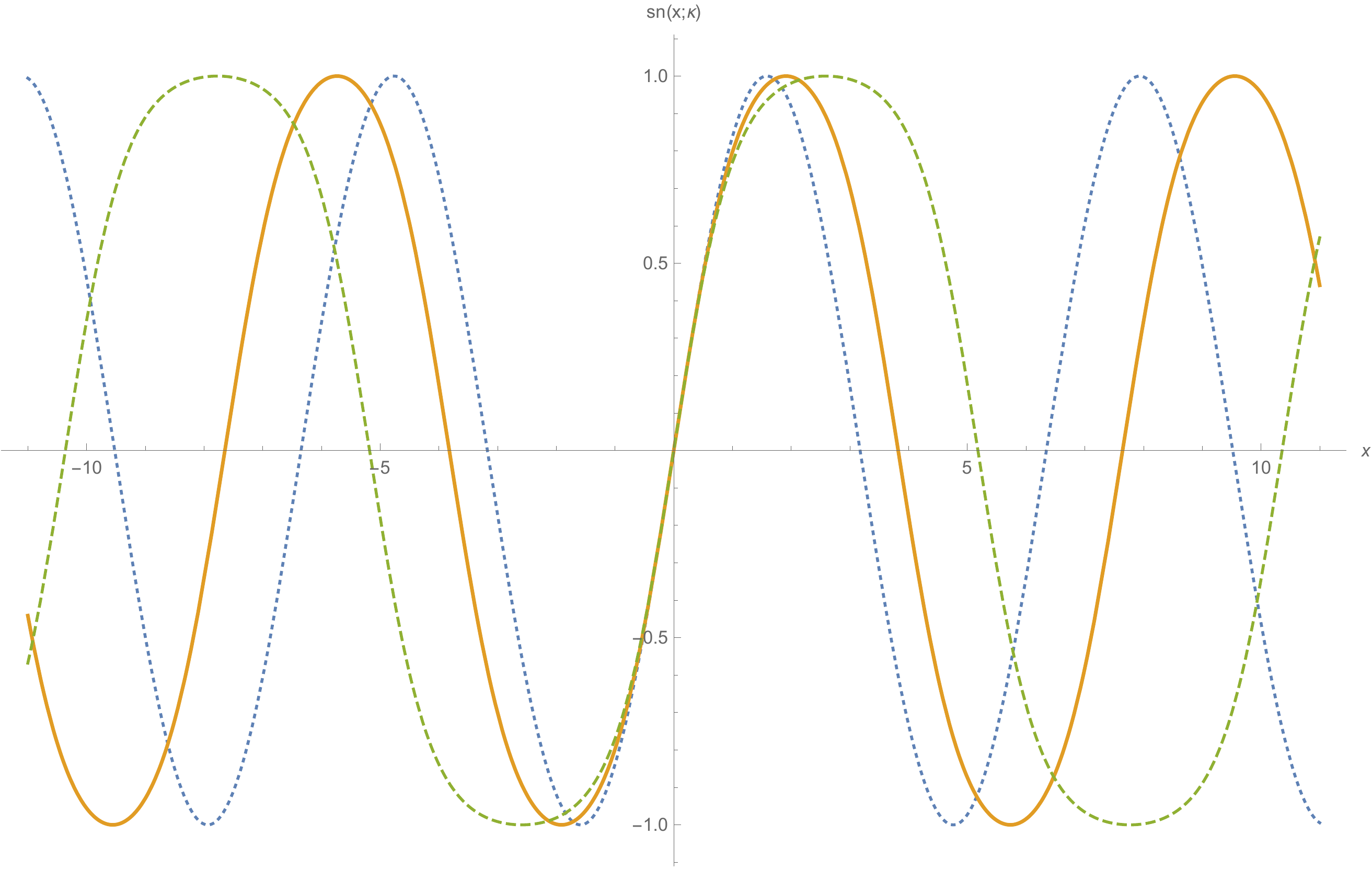}
   \caption{Snoidal wave for different values of $\kappa$. The dotted blue line represents $\mathrm{sn}(x;\kappa)$ for $\kappa=0.2$. On the thick orange line we have $\mathrm{dn}(x;\kappa)$ for $\kappa=0.75$. Finally, the dashed green line represents $\mathrm{dn}(x;\kappa)$ for $\kappa=0.95$. Notice that all of them are spatially-odd and have two zeros in each period.}\label{Fig:1}
\end{figure}

\medskip

\section{Existence of a smooth curve of periodic standing waves: Complex-valued case}\label{curve_c}

Our goal now is to establish the existence of a smooth curve of complex-valued periodic standing wave solutions to equation \eqref{phif}. Specifically, in this case we seek for solutions of the form $u(t,x)=e^{ict}\phi_c(x)$ where $\phi_c$ is assumed to be a real-valued function. Thus, by plugging this specific form of $u(t,x)$ into \eqref{phif} we obtain that if $u(t,x)$ is a solution to the equation, then $\phi_c$ must satisfy: \begin{align}\label{complex_solit_eq}
\phi_c''+(1+c^2)\phi_c-\phi_c^3=0.
\end{align}
Hence, by multiplying both sides of the equation by $\phi_c'$, after integration we obtain the following first-order differential equation in quadrature form: \begin{align}\label{complex_eliptic}
(\phi_c')^2=\dfrac{1}{2}\big(\phi_c^4-2\omega_c\phi_c^2+4A_{\phi_c}\big)=\dfrac{1}{2}F_{\phi_c}(\phi_c),
\end{align}
where, as before, $A_{\phi_c}$ stands for the integration constant, $\omega_{c}:=1+c^2$ and the polynomial function $F_{\phi_c}$ is given by \[
F_{\phi_c}(z):=z^4-2\omega_c z^2+4A_{\phi_c}.
\]
We recall again that $A_{\phi_c}$ is a nonzero (arbitrary) constant. Now, we assume we are in the particular case where $F_{\phi_c}$ takes the form:
\[
F_{\phi_c}=(z^2-\beta_1^2)(z^2-\beta_2^2) \quad \hbox{where} \quad
\beta_1^2+\beta_2^2=2\omega_c \quad \hbox{and}\quad \beta_1^2\beta_2^2=4A_{\phi_c}.
\]
Then, without loss of generality we also assume that $\beta_1>\beta_2>0$. In this case we seek for sign changing solutions, hence $\phi_c$ satisfies $\phi_c(x)\in[-\beta_2,\beta_2]$. For the sake of simplicity we define the auxiliary variable $ \psi := \beta_2^{-1}\phi_c$. Hence, by plugging this new function into the equation we obtain that $\psi$ satisfies
\[
(\psi')^2=\dfrac{\beta_1^2}{2}\big(1-\tfrac{\beta_2^2}{\beta_1^2}\psi^2\big)(1-\psi^2).
\]
On the other hand, we already know how to solve this equation, which has snoidal wave solutions. Indeed, by using \eqref{ode_sn} we obtain that the solution to equation \eqref{phif} is given by:
\begin{align}\label{proof_sn_def}
u(t,x)=\beta_2 e^{ict}\mathrm{sn}\big(\ell x;\kappa\big) \quad \hbox{where}\quad \ell:=\tfrac{\beta_1}{\sqrt{2}}, \quad \kappa:=\tfrac{\beta_2}{\beta_1} \quad \hbox{and}\quad \beta_2 = \big(2c^2+2-\beta_1^2\big)^{1/2}.
\end{align}
We point out that in  contrast with both previous cases, now $\beta_1$ explicitly depends on the speed $c$. Moreover, recalling that $\beta_1>\beta_2$, by the relation \[
\beta_1^2=2(1+c^2)-\beta_2^2 \quad \hbox{we obtain} \quad 1+c^2<\beta_1^2<2(1+c^2).
\]
Finally, since $\mathrm{sn}(\cdot,\kappa)$ has fundamental period $4K(\kappa)$, we deduce that $u(t,x)$ has fundamental period (denoted by $T_{\mathrm{c}}$): \begin{align}\label{period_sn}
T_{\mathrm{c}}=\tfrac{4\sqrt{2}}{\beta_1}K(\kappa).
\end{align}

\begin{rem}[Range of the wavelength] In this case have the following scenarios:

\smallskip

\textbf{Case $\beta_1^2\to (1+c^2)^+$:} From \eqref{proof_sn_def} we immediately see that $\kappa(\beta_1) \to1^-$ as $\beta_1\to \omega_c$. Thus, recalling that  $K(1)=+\infty$ we obtain that $T_{\mathrm{c}}\to +\infty$.

\smallskip

\textbf{Case $\beta_1^2\to2(1+c^2)^-$:} Notice that in this case, by using formula \eqref{proof_sn_def} again, we deduce that $\kappa(\beta_1)\to0^{+}$. Then, by using that $K(0)=\tfrac{\pi}{2}$ we conclude  $T_{\mathrm{c}}\to \tfrac{2\pi}{\sqrt{\omega_c}}^+$. It is worth to notice that letting $c\to 0$ we obtain that $T_{\mathrm{c}}\to2\pi$, while by letting $c\to+\infty$ we obtain $T_{\mathrm{c}}\to0$.
\end{rem}

\begin{rem}
It is worth to notice that for any given period $L>0$ and any given speed \[
\vert c\vert \in\left(c_L,+\infty\right) \quad \hbox{where} \quad c_L^2:=\max\left\{0,\tfrac{4\pi^2}{L^2}-1\right\}
\] there exists a unique pair $(\beta_1,A_{\phi_c})\in (\sqrt{\omega_c},\sqrt{2\omega_c})\times (0,+\infty)$ such that the corresponding $\mathrm{sn}(\cdot,\cdot)$ wave solution found on \eqref{proof_sn_def} satisfies $
T_{\mathrm{c}}=L$. In particular, after some point, the shorter its period, the faster its speed.
\end{rem}

Gathering all the above information we are in position to conclude the following proposition.
\begin{prop}[Smooth curve of snoidal waves]\label{MT_SN_COMPLEX}
Let $L>0$ be arbitrary but fixed. For any speed $c$ satisfying $$\vert c\vert \in\left(c_L,+\infty\right) \quad \hbox{with}\quad c_L^2:=\max\left\{0,\tfrac{4\pi^2}{L^2}-1\right\},$$
there exists unique $\beta_1\in(\sqrt{\omega_c},\sqrt{2\omega_c})$, with $\omega_c=1+c^2$, such that the dnoidal wave solution \begin{align}\label{sn_complex_def}
u(t,x)=\beta_2e^{ict}\mathrm{sn}\big(\ell x;\kappa\big), \quad \hbox{where}\quad \ell=\dfrac{\beta_1}{\sqrt{2}}, \quad  \kappa=\dfrac{\beta_2}{\beta_1} \quad \hbox{and}\quad \beta_2^2=2\omega_c-\beta_1^2,
\end{align}
has fundamental period $T_{\mathrm{sn}}=L$ and satisfies equation \eqref{complex_solit_eq}, where $\omega_c=1+c^2$. Moreover, the map $c\mapsto u(0,x)\in H^1(\T_L)$ is smooth.
\end{prop}

\begin{proof}
The proof follows similar lines as the ones made for Propositions \ref{MT_DN_CURVE}, \ref{MT_CN_CURVE} and \ref{MT_SN_CURVE}, with some obvious modifications, and hence we omit it.
\end{proof}

\medskip

\section{Orbital Instability:  Subluminal case}\label{mt_r_t}

\subsection{Spectral analysis}

From now on and for the rest of this section, in addition to the hypotheses given in Proposition \ref{MT_SN_CURVE}  we shall assume (without loss of generality) that $c>0$. Now, we start by recalling that from Proposition \ref{MT_SN_CURVE} we have the existence of a smooth curve of explicit solutions to equation \eqref{phif} given by:
\begin{align}\label{recall_sn}
u(t,x):=\beta_2\mathrm{sn}\big(\ell(x-ct)\big) \quad \hbox{where} \quad \ell:=\tfrac{\beta_1}{\sqrt{2\omega_{\mathrm{sb}}}},\quad \kappa:=\tfrac{\beta_2}{\beta_1} \quad \hbox{and}\quad \beta_2^2:=2-\beta_1^2.
\end{align}
Now, for any given $c\in(0,1)$, for the sake of clarity we shall denote by $\vec{\phi}_c^{\mathrm{sn}}$ the vector solution associated to \eqref{recall_sn}, while we shall write $\phi_c^{\mathrm{sn}}$ to refer to its first component. It is worth to notice that the equation solved by the snoidal traveling wave solution can be rewritten in terms of the functional $\mathcal{E}$ and $\mathcal{P}$ as: \[
\mathcal{E}'\big(\vec{\phi}_c^{\mathrm{sn}}\big)+c\mathcal{P}'\big(\vec{\phi}_c^{\mathrm{sn}}\big)=0,
\]
where $\mathcal{E}'$ and $\mathcal{P}'$ denote the Frechet derivates of $\mathcal{E}$ and $\mathcal{P}$ in $H^1_{\mathrm{per}}\times L^2_{\mathrm{per}}$ respectively. Then, the linearized Hamiltonian around $\vec{\phi}_c^{\mathrm{sn}}$ is given by: \begin{align}\label{op_vec_sn}
\vec{\mathcal{L}}_{\mathrm{sn}}:=\big(\mathcal{E}''+c\mathcal{P}''\big)(\vec{\phi}_c^{\mathrm{sn}})=\left(\begin{matrix}
-\partial_x^2-1+3 (\phi_c^{\mathrm{sn}})^2 & -c\partial_x \\ c\partial_x & 1
\end{matrix}\right)
\end{align}
It is worth to notice that $\vec{\mathcal{L}}_{\mathrm{sn}}$ can be regarded as a bounded self-adjoint operator defined on \[
\vec{\mathcal{L}}_{\mathrm{sn}}: H^2(\T_L)\times H^1(\T_L)\subset L^2(\T_L)\times L^2(\T_L)\to L^2(\T_L).
\]
Moreover, notice that with this definition it immediately follows that $\vec{\mathcal{L}}_{\mathrm{sn}}(\vec{\phi}_{c,x}^{\mathrm{sn}})\equiv 0$. On the other hand, the quadratic form $Q_{\mathrm{sn}}$ associated to $\vec{\mathcal{L}}_{\mathrm{sn}}$ is given by: \begin{align*}
Q_{\mathrm{sn}}&:=\langle\vec{\mathcal{L}}_{\mathrm{sn}}(\phi_1,\phi_2),(\phi_1,\phi_2)\rangle=\int \big(\phi_{1,x}^2-\phi_{1}^2+3 (\phi_c^{\mathrm{sn}})^2 \phi_1^2 +2c\phi_{1,x}\phi_2+\phi_2^2\big)dx
\\ & \ =\int \big((1-c^2)\phi_{1,x}^2-\phi_1^2+3(\phi_c^{\mathrm{sn}})^2\phi_1^2\big)dx+\int \big(c\phi_{1,x}+\phi_2\big)^2dx.
\end{align*}
Now, notice that from the first integral term of the latter identity we recognize the quadratic form associated to the operator \begin{align}\label{operator_sn_scalar}
\mathcal{L}_{\mathrm{sn}}:=-(1-c^2)\partial_x^2-1+ 3(\phi_c^{\mathrm{sn}})^2.
\end{align}

\begin{prop}\label{spectra_sn}
Under the assumptions of Proposition \ref{MT_SN_CURVE}, the operator $\mathcal{L}_{\mathrm{sn}}$ given in \eqref{operator_sn_scalar} defined on $L^2(\T_L)$ with domain $H^2(\T_L)$ defines a bounded self-adjoint operator with a unique negative eigenvalue. Moreover, zero is the second eigenvalue, which is simple, and the rest of the spectrum is discrete and bounded away from zero.
\end{prop}

\begin{proof}
First of all notice that from Weyl's essential spectral Theorem together with the compact embedding $H^2(\T_L)\hookrightarrow L^2(\T_L)$ it follows that the essential spectra of $\mathcal{L}_{\mathrm{sn}}$ is empty. Moreover, from compact self-adjoint operator theory it also follows that $\mathcal{L}_{\mathrm{sn}}$ has only point spectra, more specifically, the spectra of $\mathcal{L}_{\mathrm{sn}}$ is given by a countable infinite set of real numbers \[
\lambda_0\leq \lambda_1\leq\lambda_2\leq...
\]
satisfying that $\lambda_n\to+\infty$ as $n\to+\infty$. Therefore, the problem is reduced to study the periodic eigenvalue problem: \begin{align}\label{eig_sn_i}
\begin{cases}
\mathcal{L}_{\mathrm{sn}}f=\lambda f
\\ f(0)=f(L), \ \, f'(0)=f'(L).
\end{cases}
\end{align}
We point out that since the latter problem defines a second-order ODE, it might have (at most) two linearly independent solutions, in which case we have \emph{coexistence}, and hence, we have a double eigenvalue $\lambda_i=\lambda_{i+1}$. Then, by using the transformation $x\mapsto \ell^{-1} x$ and after some trivial re-arrangements, the latter eigenvalue problem is equivalent to solve the following (well-known) periodic problem \begin{align}\label{sn_lame}\begin{cases}
\dfrac{d^2y}{dx^2}+\big(\sigma-6\kappa^2 \mathrm{sn}^2(x)\big)y=0, 
\\ y(0)=y(4K), \ y'(0)=y'(4K),
\end{cases}
\end{align}
where the eigenvalue $\lambda\in\R$ of problem \eqref{eig_sn_i} is related to $\sigma\in\R$ by the equation:\begin{align}
\label{sigma_sn_lambda}\sigma:=\omega^{-1}\ell^{-2}(1+\lambda).
\end{align}
We recall that equation \eqref{sn_lame} is called the \emph{Jacobian form of Lame equation}. On the other hand, the latter problem has the advantage of having some well-known eigenvalues. In fact, the second and third eigenvalue of equation \eqref{sn_lame} are associated to the following values of $\sigma$ (respectively): \[
\sigma_1:=1+\kappa^2 \quad \hbox{and}\quad \sigma_2:=1+4\kappa^2.
\]
Moreover, these two eigenvalues have associated eigenfunctions given by (respectively): \begin{align}\label{eigen_functions}
Y_1:=\mathrm{cn}(x)\mathrm{dn}(x) \quad \hbox{and}\quad Y_2:=\mathrm{sn}(x)\mathrm{dn}(x).
\end{align}
We refer to \cite{An,In,MaWi} for these well-known facts. Notice also that each of these functions has exactly two zeros on the interval $[0,4K)$, and hence, by applying the inverse transformation $x\mapsto \ell x$  we infer that the corresponding transformations of $Y_1$ and $Y_2$ have exactly two zeros in $[0,L)$. Therefore, by Floquet Theory, we deduce that $\sigma_1$ and $\sigma_2$ are associated to the second and third eigenvalues of \eqref{eig_sn_i} by relation \eqref{sigma_sn_lambda}. Moreover, notice that, on the one-hand, by using \eqref{sigma_sn_lambda}  we obtain that $\sigma_1$ corresponds to the case $\lambda=0$. While on the other hand, by using relation \eqref{sigma_sn_lambda} again, we have that $\sigma_2$ is associated to \begin{align}\label{third_eig}
\lambda=\omega\ell^2\sigma_2-1=3\big(1-\tfrac{1}{2}\beta_1^2\big)>0 \quad \hbox{for}\quad \beta_1\in(1,\sqrt{2}).
\end{align}
Thus, zero is the second eigenvalue of \eqref{eig_sn_i}, what concludes the proof.
\end{proof}
As an application of the previous proposition we are able to obtain the main spectral information of $\vec{\mathcal{L}}_{\mathrm{sn}}$ required to apply Grillakis-Shatah-Strauss result.
\begin{cor}
Under the assumptions of Proposition \ref{MT_SN_CURVE} the following holds: The operator $\vec{\mathcal{L}}_{\mathrm{sn}}$ given in \eqref{op_vec_sn} defined in $L^2(\T_L)\times L^2(\T_L)$ with domain $H^2(\T_L)\times H^1(\T_L)$ defines a bounded self-adjoint operator. Moreover, its first three eigenvalues are simple, being the second one equals to zero, with associated eigenfunction given by $\vec{\phi}_{c,x}^{\mathrm{sn}}$. Additionally, the first eigenvalue is the only negative one, and the remaining part of the spectra is positive and bounded away from zero. 
\end{cor}

\begin{proof}
The proof is somehow trivial once Proposition \ref{spectra_sn} has been established, however, for the sake of completeness we show its most important steps. In fact, first of all, for the sake of simplicity, from now on we shall write $X$ to refer to the space $X:=H^1(\T_L)\times L^2(\T_L)$. Now, notice that since we are more interested in the signs of these eigenvalues rather than in their specific values, we can use the min-max principle which is particularly useful for comparing eigenvalues (see for instance \cite{RS}). In fact, let us denote by $\lambda_0,\lambda_1,\lambda_2\in\R$ the first three eigenvalue of $\vec{\mathcal{L}}_{\mathrm{sn}}$ respectively. Additionally, let us denote by $Y_0$ the eigenfunction associated to the first eigenvalue of $\mathcal{L}_{\mathrm{sn}}$ given by Proposition \ref{spectra_sn}. Then, by the min-max principle we have \begin{align}\label{minimax_sn_second}
\lambda_1=\sup_{(\psi_1,\psi_2)\in X}\inf_{\substack{(\phi_1,\phi_2)\in X\setminus\{\vec{0}\},\\ \langle (\phi_1,\phi_2),(\psi_1,\psi_2)\rangle =0}}\dfrac{\langle \vec{\mathcal{L}}_{\mathrm{sn}}(\phi_1,\phi_2),(\phi_1,\phi_2)\rangle}{\Vert (\phi_1,\phi_2)\Vert_{X}}.
\end{align}
Now, we recall that due to the spectral properties of $\mathcal{L}_{\mathrm{sn}}$ given in Proposition \ref{spectra_sn} it immediately follows that for any $\phi\in H^1(\T_L)$ it holds: \[
\langle \phi,Y_0\rangle=0 \ \implies \ \langle \mathcal{L}_{\mathrm{sn}}\phi,\phi\rangle \geq 0. 
\]
Thus, by choosing $(\psi_1,\psi_2)=(Y_0,0)\in X$ in \eqref{minimax_sn_second} we deduce that \[
\lambda_1\geq 0.
\]
Now, on the one-hand, we know that $\vec{\phi}_{c,x}^{\mathrm{sn}}$ satisfies that $\vec{\mathcal{L}}_{\mathrm{sn}}\vec{\phi}_{c,x}^{\mathrm{sn}}\equiv 0$ as well as $\langle \vec{\phi}_{c,x}^{\mathrm{sn}},(Y_0,0)\rangle=0$, while on the other hand,  \[
\langle \vec{\mathcal{L}}_{\mathrm{sn}}(Y_0,0),(Y_0,0)\rangle=\langle \mathcal{L}_{\mathrm{sn}}Y_0,Y_0\rangle<0.
\]
Therefore, gathering all the previous information it follows that $\lambda_1=0$ and that $\lambda_0<0$. Finally, by using again the min-max principle we know that $\lambda_2 $ is given by \[
\lambda_2=\sup_{\substack{(\psi_1,\psi_2)\in X, \\ (\psi_3,\psi_4)\in X}}\inf_{\substack{(\phi_1,\phi_2)\in X\setminus\{\vec0\},\\ \langle (\phi_1,\phi_2),(\psi_1,\psi_2)\rangle =0,\\ \langle (\phi_1,\phi_2),(\psi_3,\psi_4)\rangle =0}}\dfrac{\langle \vec{\mathcal{L}}_{\mathrm{sn}}(\phi_1,\phi_2),(\phi_1,\phi_2)\rangle}{\Vert (\phi_1,\phi_2)\Vert_{X}}.
\]
Thus, in the same fashion as before, by choosing $(\psi_1,\psi_2)=(Y_0,0)$ and $(\psi_3,\psi_4)=(\phi_{c,x}^{\mathrm{sn}},0)$ together with Proposition \ref{spectra_sn} it follows that $\lambda_2>0$, what concludes the proof. 
\end{proof}

\subsection{Orbital Instability}
Finally, we are ready to prove our orbital instability result for snoidal traveling waves solutions. In fact, as we discussed before, in order to show the Instability Theorem, we shall apply the Grillakis-Shatah-Strauss classical result (see \cite{GSS}). In fact, once the existence of the smooth curve of traveling waves solutions and the main spectral information of the linearized Hamiltonian around $\vec{\phi}_c^{\mathrm{sn}}$ are established (see Proposition \ref{MT_SN_CURVE} and \ref{spectra_sn} respectively), the problem is reduced to study the convexity/concavity of the scalar function:
\[
d(c):=\big(\mathcal{E}+c\mathcal{P}\big)(\vec{\phi}_c^{\mathrm{sn}}).
\]
We recall that under our current hypothesis, the snoidal wave $\vec{\phi}_c^{\mathrm{sn}}$ is orbitally stable if and only if $d(c)$ is convex. In other words, if and only if $d''(c)>0$. Moreover, recalling that $\vec{\phi}_{c}^{\mathrm{sn}}$ is a critical point of the \emph{action functional} $\mathcal{E}+c\mathcal{P}$, we deduce that \begin{align}\label{d_prime_sn}
d'(c)=-c\int_0^L (\phi_{c,x}^{\mathrm{sn}})^2dx.
\end{align}
Since we still have to compute the next derivative of $d$, before going further it is convenient to establish a formula for the latter integral in terms of functions with well-known monotonicity properties. In fact, by using \eqref{subluminal_int} it follows that \[
\int (\phi_{c,x}^{\mathrm{sn}})^2=\dfrac{1}{2\omega_{\mathrm{sb}}}\int \big((\phi_c^{\mathrm{sn}})^4-2( \phi_c^{\mathrm{sn}})^2+\beta_1^2(2-\beta_1^2)\big)dx. 
\]
Thus, by using the relation between $\beta_1$ and $\kappa$ as well as identity \eqref{period_sn}, direct computations yield us to 
\begin{align}
\int_0^L \big(\phi_{c,x}^{\mathrm{sn}}\big)^2&=\dfrac{2\sqrt{2}\beta_1}{3\sqrt{\omega_{\mathrm{sb}}}}\Big(%
\beta_1^2(2+\kappa^2)K(\kappa)-2\beta_1^2(1+\kappa^2)E(\kappa)+6E(\kappa)-3\beta_1^2K(\kappa)%
\Big)\nonumber
\\ & = \dfrac{32}{3L}\Big(%
E(\kappa)+(1-\beta_1^2)K(\kappa)%
\Big)K(\kappa), \label{c_derivative_sn}
\end{align}
where we have used the well-known formulas (see for instance identities $(310.02)$ and $(310.04)$ in \cite{BiFr}): \begin{align}
\int_0^K\mathrm{sn}^2(x;\kappa)dx&=\dfrac{1}{\kappa^2}\Big(K(\kappa)-E(\kappa)\Big),\label{sn_l2norm_formula}
\\ \int_0^K \mathrm{sn}^4(x;\kappa)dx&=\dfrac{1}{3\kappa^4}\Big((2+\kappa^2)K(\kappa)-2(1+\kappa^2)E(\kappa)\Big).\nonumber
\end{align}
Finally, in order to compute the derivative of \eqref{c_derivative_sn} with respect to $c$, we shall need an expression for the derivative of $\beta_1(c)$. In fact, differentiating \eqref{sb_period} with respect to $c$, and recalling that $L$ is fixed, we deduce that
\[
\dfrac{d\beta_1}{dc}=\dfrac{4\sqrt{2}}{L}\Big(\sqrt{1-c^2}K'\dfrac{d\kappa}{d\beta_1}\dfrac{d\beta_1}{dc}-\dfrac{c}{\sqrt{1-c^2}}K\Big).
\]
Thus, by re-arranging terms we get \[
\Big(\dfrac{4\sqrt{2\omega_{\mathrm{sb}}}}{L}K'(\kappa)\dfrac{d\kappa}{d\beta_1}-1\Big)\dfrac{d\beta_1}{dc}=\dfrac{4\sqrt{2}c}{L\sqrt{\omega_\mathrm{sb}}}K(\kappa).
\]
On the other hand, we recall that 
\[
\dfrac{dK}{d\kappa}=\dfrac{E(\kappa)-(1-\kappa^2)K}{\kappa(1-\kappa^2)} \quad \hbox{and} \quad \dfrac{d\kappa}{d\beta_1}=-\dfrac{2}{\beta_1^2\sqrt{2-\beta_1^2}}.
\]
Therefore, gathering the previous identities, recalling that $K'>0$, we conclude that $\tfrac{d\beta_1}{dc}<0$. Hence, once the sign of $\beta_1'$ has been found, we are able to infer the monotonicity of \eqref{c_derivative_sn}. In fact, by direct differentiation with respect to $c$, we obtain
\begin{align*}
\dfrac{3L}{32}\dfrac{d}{dc}\int_0^L\big(\phi_{c,x}^{\mathrm{sn}}\big)^2&= \left(E'(\kappa)\dfrac{d\kappa}{d\beta_1}-2\beta_1K(\kappa)+(1-\beta_1^2)K'(\kappa)\dfrac{d\kappa}{d\beta_1}\right)\dfrac{d\beta_1}{dc}K(\kappa)
\\ & \quad +\Big(E(\kappa)+(1-\beta_1^2)K(\kappa)\Big)K'(\kappa)\dfrac{d\kappa}{d\beta_1}\dfrac{d\beta_1}{dc}
\\ &= \bigg[\dfrac{\kappa'}{\kappa(1-\kappa^2)} E^2(\kappa)+\dfrac{2(1-\beta_1^2)\kappa'}{\kappa(1-\kappa^2)}E(\kappa)K(\kappa)
\\ & \quad -\left(\dfrac{\kappa'}{\kappa}+2\beta_1+\dfrac{2(1-\beta_1^2)\kappa'}{\kappa}\right)K^2(\kappa)\bigg]\dfrac{d\beta_1}{dc}
\\ &=: \Big(\mathrm{I}\cdot E^2(\kappa)+\mathrm{II}\cdot E(\kappa)K(\kappa)+\mathrm{III}\cdot K^2(\kappa)\Big)\dfrac{d\beta_1}{dc}.
\end{align*}
Thus, in order to deduce the sign of the previous expression, we split the analysis into two steps. First of all, we rewrite the second term $\mathrm{II}$ as: \[
\mathrm{II}=\dfrac{\kappa'(1-\kappa^2)(1-\beta_1^2)}{\kappa(1-\kappa^2)}-\dfrac{\kappa'(1-\kappa^2)}{\kappa(1-\kappa^2)}=:\mathrm{A}_1+\mathrm{A}_2.
\]
Now, on the one-hand, recalling that $\kappa'=\tfrac{d\kappa}{d\beta_1}<0$ for all $\beta_1\in(1,\sqrt{2})$, together with the fact that $K'(\kappa)>0$ for all $\beta_1\in(1,\sqrt{2})$, we deduce that \[
\mathrm{I}\cdot E^2(\kappa)+\mathrm{A}_2E(\kappa)K(\kappa)=\kappa'E(\kappa)\left(\dfrac{E(\kappa)-(1-\kappa^2)K(\kappa)}{\kappa(1-\kappa^2)}\right)<0.
\]
On the other hand, noticing that $\mathrm{A}_1>0$ and recalling that $K(\kappa)>E(\kappa)$ we get \[
\mathrm{A}_1E(\kappa)K(\kappa)+\mathrm{III}\cdot K^2(\kappa)\leq \left(\mathrm{A}_1+\mathrm{III}\right)K^2(\kappa)=\dfrac{2(1-\beta_1^2)}{\beta_1} K^2(\kappa) < 0.
\]
Therefore, we conclude that for all $c\in(0,1)$ it holds: \[
\dfrac{3L}{32}\dfrac{d}{dc}\int_0^L\big(\phi_{c,x}^{\mathrm{sn}}\big)^2=\Big(\mathrm{I}\cdot E^2(\kappa)+\mathrm{II}\cdot E(\kappa)K(\kappa)+\mathrm{III}\cdot K^2(\kappa)\Big)\dfrac{d\beta_1}{dc}>0.
\]
Finally, noticing that \[
d''(c)\leq-c\dfrac{d}{dc}\int_0^L\big(\phi_{c,x}^{\mathrm{sn}}\big)^2dx,
\]
we obtain that $d''(c)<0$ for all $c\in(0,1)$, what concludes the proof by applying the instability result in \cite{GSS}. Specifically, we obtain the following result. 

\begin{thm}
Under the assumptions of Proposition \ref{MT_SN_CURVE}, the snoidal wave solution given by \eqref{sn_sol_def} is orbitally unstable under the periodic flow of the $\phi^4$-equation \eqref{phif}.
\end{thm}

\medskip

\section{Orbital Stability:  Real-valued stationary case}\label{mt_r_s}

Within this section we shall assume that $L>2\pi$. Hence, by Proposition \ref{MT_SN_CURVE} we have the explicit real-valued time-independent periodic solution to equation \eqref{phif}, which is given by: \[
S(t,x)\equiv S(x):=\beta_2\mathrm{sn}\big(\ell x\big) \quad \hbox{where} \quad \ell:=\dfrac{\beta_1}{\sqrt{2}},\quad \kappa:=\dfrac{\beta_2}{\beta_1} \quad \hbox{and}\quad \beta_2^2:=2-\beta_1^2,
\]
where in this case $\beta_1\in(1,\sqrt{2})$ is uniquely defined once $L>2\pi$ is fixed. Recall that from the analysis made in the previous section it follows that these solution are orbitally unstable. However, in this section we shall prove that under some additional hypothesis it is still possible to obtain an orbital stability result. For this purposes we shall follow the strategy in \cite{HPW,KMM}. We point out that these additional hypothesis are not directly transferable to the non-zero speed case.

\medskip

One important advantage in this case is given by the preservation of the spatial-oddness by the periodic flow of the $\phi^4$-equation. That is, if the initial data is $(\mathrm{odd},\mathrm{odd})$, then so is the solution for all times. Then, recalling that $\mathrm{sn}(x)$ is odd, we obtain that if the initial perturbation $\vec\varepsilon_0=(\varepsilon_{0,1},\varepsilon_{0,2})=(\mathrm{odd},\mathrm{odd})$, then so is the solution associated to \[
(\phi_{0,1},\phi_{0,2})=(S,0)+(\varepsilon_{0,1},\varepsilon_{0,2}).
\]
Thus, it is natural to study the time evolution of an initial odd perturbation of $(S,0)$ in terms of the evolution of its perturbation $\vec{\varepsilon}(t)$. In other words, for all times we shall write the solution as $\vec{\phi}(t,x)=(S(x),0)+\vec{\varepsilon}(t,x)$.  Moreover, by using equation \eqref{phif_2} we deduce that $\vec{\varepsilon}(t,x)$ satisfy the first-order system
\begin{align}\label{eps_system}
\begin{cases}
\partial_t\varepsilon_1=\varepsilon_2,
\\ \partial_t\varepsilon_2=-\mathcal{L}_{\mathrm{s}}\varepsilon_1-3S\varepsilon_1^2-\varepsilon_1^3,
\end{cases}
\end{align}
where $\mathcal{L}_s$ is the linearize operator around $S$, which is given by: \[
\mathcal{L}_s=-\partial_x^2-1+3S^2.
\]
From the energy conservation of \eqref{phif} it follows that system \eqref{eps_system} has the following conservation law: \begin{align}\label{eps_energy}
\widetilde{\mathcal{E}}(\vec{\varepsilon}(t)):=\langle \mathcal{L}_s\varepsilon_1,\varepsilon_1\rangle+\int \varepsilon_2^2+2\int S\varepsilon_1^3+\dfrac{1}{2}\int \varepsilon_1^4=\widetilde{\mathcal{E}}(\vec{\varepsilon}_0).
\end{align}
Now, on the one-hand, from the spectral analysis developed in the latter section, we know that there is only one negative eigenvalue associated with the operator $\mathcal{L}_s$. Even more, due to the sign property satisfied by $Y_0$ (the eigenfunction associated to this negative direction), by standard Floquet Theory (see for instance \cite{MaWi}) we know that $Y_0$ is an even function. Furthermore, $S'(x)$, which is associated to the second eigenvalue $\lambda=0$, is also even. On the other hand, notice that since $Y_0$ and $S'$ are even regarded as functions defined on the whole line $\R$, and since they are also $L$-periodic at the same time, it follows that they are even with respect to $x=\tfrac{1}{2}L$. Of course, the same remark also holds for odd functions. Therefore, since $[0,L]$ is symmetric with respect to $x=\tfrac{1}{2}L$, it follows that odd and even functions (with respect to the whole line) belonging to $H^1(\T_L)$ are orthogonal in the associated $H^1(\T_L)$-inner product.

\medskip

Gathering all the previous analysis we are in position to establish the following lemma.

\begin{lem}\label{coercivity}
Under the assumptions and notations of Proposition \ref{MT_SN_CURVE}, for any odd function $\upsilon\in H^1(\T_L)$ the following holds: \[
\langle \mathcal{L}_{s}\upsilon,\upsilon\rangle \geq \lambda^2\Vert \upsilon\Vert_{H^1_{\mathrm{per}}}^2 \quad \hbox{where} \quad \lambda^2=\tfrac{3\beta_2^2}{4+3\beta_2^2}.
\]
\end{lem}

\begin{proof}
In fact, by using Proposition \ref{spectra_sn}, noticing that $3(1-\tfrac{1}{2}\beta_1^2)=\tfrac{3}{2}\beta_2^2$ (see \eqref{third_eig}), and the eveness of the first two eigenfunctions, by the Spectral Theorem we deduce that, for any odd function $\upsilon\in H^1(\T_L)$, we have \begin{align}\label{coer_st}
\langle\mathcal{L}_s\upsilon,\upsilon\rangle \geq \tfrac{3}{2}\beta_2^2\Vert \upsilon\Vert_{L^2}^2.
\end{align}
Now, we shall prove that by lowering the constant $\tfrac{3}{2}\beta_2^2$ we can \emph{improve} the latter inequality in the sense that we can change the $L^2$ by the $H^1$ norm. In fact, let us start by rewriting the quadratic form in a more convenient way:
\begin{align}\label{stationary_quadratic}
\langle\mathcal{L}_s\upsilon,\upsilon\rangle &=\int \upsilon_x^2+3\int S^2\upsilon^2-\int \upsilon^2=\int \upsilon_x^2+(3\beta_2^2-1)\int \upsilon^2-3\beta_2^2\int \upsilon^2(1-\mathrm{sn}^2(x)) .
\end{align}
Then, consider $\alpha,\eta\in\R$ given by: \[
\alpha:=\dfrac{4}{4+3\beta_2^2}, \quad \hbox{and}\quad \eta:=\dfrac{1-\alpha}{3\beta_2^2-1}.
\]
Notice that $1-\alpha>0$. We point out that we have chosen $\alpha$ and $\eta$ so that we have (in particular) the following relations (which can be verify by direct evaluation): \[
\big(1-\alpha+\eta\big)(3\beta_2^2-1)=3\beta_2^2(1-\alpha)=\dfrac{9\beta_2^4}{4+3\beta_2^2}>0.
\]
Thus, we can rewrite identity \eqref{stationary_quadratic} again as: \begin{align}
\langle\mathcal{L}_s\upsilon,\upsilon\rangle&=(1-\alpha)\int\upsilon_x^2-\eta(3\beta_2^2-1)\int \upsilon^2+\dfrac{9\beta_1^4}{4+3\beta_1^2}\int \mathrm{sn}^2(x)\upsilon^2+\alpha\langle\mathcal{L}_s\upsilon,\upsilon\rangle\nonumber
\\ & \geq(1-\alpha)\int\upsilon_x^2+\left(\dfrac{3\alpha\beta_2^2}{2}-\eta(3\beta_2^2-1)\right)\int \upsilon^2,\label{ineq_proof_coer}
\end{align}
where in the last line we have used \eqref{coer_st}. Finally, by straightforward computations we see that \[
\dfrac{3}{2}\alpha\beta_2^2-\eta(3\beta_2^2-1)=\dfrac{3\beta_2^2}{4+3\beta_2^2}=1-\alpha>0.
\]
Therefore, by plugging this latter identity into inequality \eqref{ineq_proof_coer} we obtain: \[
\langle \mathcal{L}_s\upsilon,\upsilon\rangle\geq (1-\alpha)\int \big(\upsilon_x^2+\upsilon^2\big),
\]
what finish the proof of the lemma.
\end{proof}

With the above information we are in position to establish our orbital stability result.

\begin{thm}
Under the assumptions of Proposition \ref{MT_SN_CURVE}, assuming additionally $L>2\pi$, there exists $\delta>0$ small enough such that for any initial data \[
\vec{\varepsilon}_0=(\varepsilon_{0,1},\varepsilon_{0,2})\in H^1(\T_L)\times L^2(\T_L) \quad \,\hbox{with} \,\quad (\varepsilon_{0,1},\varepsilon_{0,2}) =(\mathrm{odd},\mathrm{odd}),
\]
satisfying $\Vert (\varepsilon_{0,1},\varepsilon_{0,2})\Vert_{H^1_{\mathrm{per}}\times L^2_{\mathrm{per}}}\leq\delta$, then the following holds: There exists a constant $C>0$ such that the solution to equation \eqref{eps_system} associated to $\vec{\varepsilon}_0$ satisfies \[
\hbox{for all }\,t\in\R,\quad \Vert (\varepsilon_1,\varepsilon_2)(t)\Vert_{H^1_{\mathrm{per}}\times L^2_{\mathrm{per}}}\leq C\delta.
\]
\end{thm}

\begin{proof}
In fact, noticing that the periodic flow of the system \eqref{eps_system} (as well as the periodic $\phi^4$ flow) preserves the oddness of the initial data, we have that for all $t\in\R$ the solution satisfies \[
\vec{\varepsilon}(t)=(\varepsilon_1,\varepsilon_2)(t)=(\mathrm{odd},\mathrm{odd})(t).
\]
Hence, by plugging the result given by Lemma \ref{coercivity} into the explicit form of $\widetilde{\mathcal{E}}$ (see \eqref{eps_energy}) it immediately follows that: \[
\widetilde{\mathcal{E}}(\vec{\varepsilon}(t))\geq \lambda^2\big(\Vert \varepsilon_2(t)\Vert_{L^2_{\mathrm{per}}}^2+\Vert \varepsilon_1(t)\Vert_{H^1_{\mathrm{per}}}^2\big)-O\big(\Vert \varepsilon_1(t)\Vert_{H^1_{\mathrm{per}}}^3\big).
\]
Recalling that $\widetilde{\mathcal{E}}$ is conserved along the trajectory, in order to conclude it is enough to notice that \[
\widetilde{\mathcal{E}}(\vec{\varepsilon}_0(t))\leq \Vert \varepsilon_{0,2}\Vert_{L^2_{\mathrm{per}}}^2+\Vert \varepsilon_{0,1}\Vert_{H^1_{\mathrm{per}}}^2+O\big(\Vert \varepsilon_{0,1}\Vert_{H^1_{\mathrm{per}}}^3\big).
\]
Thus, by making $\delta>0$ small enough, the proof follows by gathering both inequalities.
\end{proof}

\medskip

\section{Orbital Stability: Complex-valued case}\label{os_c}

\subsection{Spectral analysis}
From now on and for the rest of this section, in addition to the hypothesis given in Proposition \ref{MT_SN_COMPLEX}, we shall assume (without loss of generality) that $c>0$. Now, let us start by recalling that from Proposition \ref{MT_SN_COMPLEX} we have the existence of a smooth curve of complex-valued explicit solutions to equation \eqref{phif} given by: \begin{align}\label{recall_complex_snoidal}
u(t,x)=\beta_2e^{ict}\mathrm{sn}\big(\ell x;\kappa\big), \quad \hbox{where}\quad \ell=\dfrac{\beta_1}{\sqrt{2}}, \quad  \kappa=\dfrac{\beta_2}{\beta_1} \quad \hbox{and}\quad \beta_2^2=2\omega_c-\beta_1^2.
\end{align}
We also recall that in this case $\omega_c=1+c^2$ and $\beta_1\in(\sqrt{\omega_c},\sqrt{2\omega_c})$. On the other hand, in order to avoid misunderstandings with the previous case, from now on, for any given speed $c\in(0,+\infty)$, we denote by $\psi_c^{\mathrm{sn}}$ the function given by the relation $u(t,x)=e^{ict}\psi_c^{\mathrm{sn}}(x)$. Additionally, we shall denote by $\vec{\psi}_c^{\mathrm{sn}}$ to refer to the vector\footnote{See notation \eqref{complex_notation} below.} \[
\vec{\psi}_{c}^{\mathrm{sn}}=(\psi_c^{\mathrm{sn}},c\psi_c^{\mathrm{sn}},0,0).
\]
One of the most important differences with respect to the previous case is that, due to the complex character of the solution, in this case we have to write $\vec\phi$ as a $4$ dimensional vector. Specifically, from now on we shall write $\vec\phi$ as \begin{align}\label{complex_notation}
\vec{\phi}=\big(\mathrm{Re}\phi_1,\mathrm{Im}\phi_2,\mathrm{Im}\phi_1,\mathrm{Re}\phi_2\big)
\end{align}
We remark that the coordinates of $\vec{\phi}$ are not in the most intuitive order. On the other hand, it is worth to notice that in this case we can rewrite equation \eqref{phif} as \[
\partial_t\vec{\phi}=\mathbf{J}\mathcal{E}'(\vec{\phi}) \quad \hbox{where}\quad \mathbf{J}:=\left(\begin{matrix}
0 & 0 & 0 & 1
\\ 0 & 0 & -1 & 0
\\ 0 & 1 & 0 & 0
\\ -1 & 0 & 0 & 0
\end{matrix}\right).
\]
The most important advantage of rewriting $\vec{\phi}$ in this strange order is to ease the spectral analysis for the linearize Hamiltonian $\mathcal{E}''-c\mathcal{F}''$. In fact, this is not an arbitrary choice of coordinates and has been used several times before, we refer (for instance) to \cite{AnNa2} for a previous use of these coordinates in a similar context. Of course, by writing $\vec\phi=(\mathrm{Re}\phi_1,\mathrm{Im}\phi_1,\mathrm{Re}\phi_2,\mathrm{Im}\phi_2)$ all the below properties shall also hold. Now, it is worth to notice that the equation solved by the complex-valued snoidal standing wave solution (see \eqref{complex_solit_eq}) can be rewritten in terms of the functional $\mathcal{E}$ and $\mathcal{F}$ as: \[
\mathcal{E}'\big(\vec{\psi}_c^{\mathrm{sn}}\big)- c\mathcal{F}'\big(\vec{\psi}_c^{\mathrm{sn}}\big)=0.
\]
In other words, the snoidal solution is a critical point of the functional $\mathcal{E}-c\mathcal{F}$. Therefore, by following Grillakis-Shatah-Strauss result, the stability/instability property follows from the study of the spectral properties of the linearized Hamiltonian around $\vec{\psi}_c^{\mathrm{sn}}$, that is, \[
\vec{\mathcal{L}}_\mathrm{sn}:=\big(\mathcal{E}''-c\mathcal{F}''\big)(\vec{\psi}_c^{\mathrm{sn}}).
\]
It is important to notice that since we are adopting notation \eqref{complex_eliptic}, the later identity defines a $4\times 4$ matrix operator. On the oher hand, in this case it is convenient to split the analysis into two parts. In fact, we define the operators:  \begin{align}\label{snisnr}
\vec{\mathcal{L}}_{\mathrm{sn},\mathcal{R}}:=\left(\begin{matrix}
-\partial_x^2-1+3 (\psi_c^{\mathrm{sn}})^2 & -c 
\\ -c & 1
\end{matrix}\right), %
\ \ \vec{\mathcal{L}}_{\mathrm{sn},\mathcal{I}}:=\left(\begin{matrix}
-\partial_x^2-1+(\psi_c^{\mathrm{sn}})^2 & c
\\ c & 1
\end{matrix}\right)
\end{align}
where $\vec{\mathcal{L}}_{\mathrm{sn},\mathcal{R}}$ and $\vec{\mathcal{L}}_{\mathrm{sn},\mathcal{I}}$ denote the \emph{real} and \emph{imaginary} parts of the main operator: \[
\vec{\mathcal{L}}_{\mathrm{sn}}=\left(\begin{matrix}
\vec{\mathcal{L}}_{\mathrm{sn},\mathcal{R}} & 0 \\ 0 & \vec{\mathcal{L}}_{\mathrm{sn},\mathcal{I}}
\end{matrix}\right),
\]
where both zeros denote the $2\times 2$ zero matrix. Now, we intend to proceed in a similar fashion as in the latter section. However, as the previous definitions suggest, in this case it is better to split the spectral analysis of $\vec{\mathcal{L}}_{\mathrm{dn}}$ into two different steps. In fact, we start by considering the quadratic form associated to $\vec{\mathcal{L}}_{\mathrm{dn},\mathcal{R}}$, which is given by  \begin{align*}
Q_{{\mathrm{sn},\mathcal{R}}}&:=\langle\vec{\mathcal{L}}_{\mathrm{sn},\mathcal{R}}(\phi_1,\phi_2),(\phi_1,\phi_2)\rangle=\int \big(\phi_{1,x}^2-\phi_1^2+3(\psi_c^{\mathrm{sn}})\phi_1^2-2c\phi_1\phi_2+\phi_2^2\big)dx
\\ & \ = \int\big(\phi_{1,x}^2-(1+c^2)\phi_1^2+3(\psi_c^{\mathrm{sn}})\phi_1^2\big)dx +\int \big(c\phi_1-\phi_2\big)^2dx
\end{align*}
In the same fashion as before, from the latter identity we can recognize the quadratic form associated to the operator \begin{align}\label{op_sn_r}
\mathcal{L}_{\mathrm{sn},\mathcal{R}}:=-\partial_x^2-(1+c^2)+3 (\psi_c^{\mathrm{sn}})^2.
\end{align}
Proceeding similarly with the quadratic form $Q_{\mathrm{sn},\mathcal{I}}$, we find the linear operator \begin{align}\label{op_sn_i}
\mathcal{L}_{\mathrm{sn},\mathcal{I}}:=-\partial_x^2-(1+c^2)+(\psi_c^{\mathrm{sn}})^2.
\end{align}
Thus, we turn our attention to study the spectral properties of $\mathcal{L}_{\mathrm{sn},\mathcal{R}}$ and $\mathcal{L}_{\mathrm{sn},\mathcal{I}}$.
\begin{prop}\label{prop_sni_1}
Under the assumptions of Proposition \ref{MT_SN_COMPLEX}, the operator $\mathcal{L}_{\mathrm{sn},\mathcal{R}}$ given in \eqref{op_sn_r} defined in $L^2(\T_L)$ with domain $H^2(\T_L)$ defines a bounded self-adjoint operator with a unique negative eigenvalue. Moreover, zero is the second eigenvalue, which is simple, and the rest of the spectrum is discrete and bounded away from zero.
\end{prop}

\begin{proof}
The proof follows similar lines as the one of Proposition \ref{spectra_sn} and hence we only sketch its main steps. In fact, from Weyl's essential Theorem, the compact embedding $H^2(\T_L)\hookrightarrow L^2(\T_L)$ and standard theory of compact operators, it follows that $\mathcal{L}_{\mathrm{sn},\mathcal{R}}$ has only point spectra. Moreover, its spectra is given by a countable infinite set of real numbers tending to infinity. On the other hand, by using the transformation $x\mapsto \ell^{-1} x$ and after some direct re-arrangements, the eigenvalue problem for $\mathcal{L}_{\mathrm{sn}}$ is equivalent to the following second order periodic equation: \begin{align}\label{sevensix}
\begin{cases}
\dfrac{d^2y}{dx}+\big(\sigma-6\kappa^2\mathrm{sn}^2(x)\big)y=0
\\ y(0)=y(4K), \ \, y'(0)=y(4K),
\end{cases}
\end{align}
where the eigenvalue $\lambda\in\R$ is related to $\sigma\in\R$ by the equation: \begin{align}\label{eq_eig_sn_complex}
\sigma=\ell^{-2}(1+c^2+\lambda).
\end{align}
Now, we recall that the previous eigenvalue equation is classical. In particular, we have that 
\[
\sigma_1:=1+\kappa^2 \quad \hbox{and}\quad \sigma_2:=1+4\kappa^2,
\]
are the second and third eigenvalues of \eqref{sevensix} respectively. Then, in the same fashion as before, by using the definitions of $\ell$ and $\kappa$ we see that, eigenvalue $\sigma_1$ is associated to the case $\lambda=0$ and $\sigma_2$ with \[
\lambda=\ell^2(1+\kappa^2)-(1+c^2)=\dfrac{\beta_1^2}{2}+2\beta_2^2-(1+c^2)=3(1+c^2)-\dfrac{3}{2}\beta_1^2>0
\]
where inequality holds for all $\beta_1\in(\sqrt{\omega_c},\sqrt{2\omega_c})$. Finally, by using the explicit form of the eigenfunctions $Y_1$ and $Y_2$ defined in \eqref{eigen_functions}, since they have exactly two zeros in $[0,L)$, we deduce that these two values of $\lambda$ correspond to the second and third eigenvalues of our original operator $\mathcal{L}_{\mathrm{sn},\mathcal{R}}$, what concludes the sketch of the proof.
\end{proof}

Regarding the spectral information of $\mathcal{L}_{\mathrm{sn},\mathcal{I}}$, due to the term $(\psi_c^{\mathrm{sn}})^2$, in this case we have two different negative directions.

\begin{prop}\label{op_sn_i_2}
Under the assumptions of Proposition \ref{MT_SN_COMPLEX}, the operator $\mathcal{L}_{\mathrm{sn},\mathcal{I}}$ given in \eqref{op_sn_i} defined in $L^2(\T_L)$ with domain $H^2(\T_L)$ defines a bounded self-adjoint operator with two different negative eigenvalues, which are both simple. Moreover, zero is the third eigenvalue, which is also simple, and the rest of the spectrum is discrete and bounded away from zero.
\end{prop}

\begin{proof}
In fact, once again, by Weyl's essential Theorem, the compact embedding $H^2(\T_L)\hookrightarrow L^2(\T_L)$ and standard theory of compact operators, it follows that $\mathcal{L}_{\mathrm{sn},\mathcal{I}}$ has only point spectra. Moreover, its spectra is given by a countable infite set of real numbers tending to infinity. Additionally, note that from equation \eqref{complex_solit_eq} we have that $\mathcal{L}_{\mathrm{sn},\mathcal{I}}\psi_c^{\mathrm{sn}}\equiv0$. However, $\psi_{c}^{\mathrm{sn}}$ has exactly two zeros in $[0,L)$, and hence, from Floquet Theory we know that it corresponds to either the second or the third eigenvalue. On the other hand, in this case we can easily find the first eigenvalue, as well as the remaining candidate to be the second/third eigenvalue. In fact, in order to rewrite the eigenvalue equation in a more standard form, by using the transformation $x\mapsto \ell^{-1} x$ and by some direct re-arrangements we have that the eigenvalue problem is equivalent to find the values of $\sigma\in\R$ for which the following equation has nontrivial solutions: \begin{align}\label{reduction_complex}
\begin{cases}
\dfrac{d^2y}{dx}+\big(\sigma-2\kappa^2\mathrm{sn}^2(x)\big)y=0
\\ y(0)=y(4K), \ \, y'(0)=y(4K),
\end{cases}
\end{align}
where the eigenvalue $\lambda\in\R$ is related to $\sigma\in\R$ by the equation: \begin{align}\label{sigma_final}
\sigma=\ell^{-2}(1+c^2+\lambda).
\end{align}
As we already said, in this case we shall explicitly find the first and second/third eigenvalue with their associated eigenfunctions. In fact, by explicit computations it is straightforward to check that $\sigma=\kappa^2$ is an (simple, by Floquet Theory) eigenvalue with associated eigenfunction $\mathrm{dn}(x;\kappa)$. Thus, by using relation \eqref{sigma_final} we deduce that $\lambda_1=-\tfrac{1}{2}\beta_1^2$ is the corresponding eigenvalue of our original operator $\mathcal{L}_{\mathrm{sn},\mathcal{I}}$. Moreover, by applying the inverse transformation $x\mapsto \ell x$, we see that $\mathrm{dn}(\ell x;\kappa)$ is the eigenfunction associated to $\lambda_1$. On the other hand, notice that $\mathrm{dn}(\ell x;\kappa)$ has no zeros in $[0,L)$, and hence, by Floquet Theory, it corresponds to the first eigenvalue of $\mathcal{L}_{\mathrm{sn},\mathcal{I}}$. In the same fashion, by explicit computations we see that $\sigma=1$ is an eigenvalue of \eqref{reduction_complex} with associated eigenfunction $\mathrm{cn}(x,\kappa)$. Moreover, by using \eqref{sigma_final} we infer that $\sigma=1$ corresponds to the case $\lambda=-\tfrac{1}{2}\beta_2^2$ and additionally we have that $\mathrm{cn}(\ell x,\kappa)$ has exactly two zeros in $[0,L)$. Therefore, $\psi_c^{\mathrm{sn}}$ correspond to the third eigenvalue of $\mathcal{L}_{\mathrm{sn},\mathcal{R}}$, what concludes the sketch of the proof.
\end{proof}

\begin{rem}
We remark that all the functions associated with negative directions found in the previous two propositions are all \textbf{even}. This fact shall be important for the proof the Theorem \ref{MT_COMPLEX}.
\end{rem}

Gathering the last two propositions we have the following information about the matrix operators 
\begin{cor}\label{final_spectral}
The matrix operator $\vec{\mathcal{L}}_{\mathrm{sn},\mathcal{R}}$ given in \eqref{snisnr} defined in $L^2(\T_L)\times L^2(\T_L)$ with domain $H^2(\T_L)\times H^1(\T_L)$ has exactly one negative eigenvalue. Furthermore, the matrix operator $\vec{\mathcal{L}}_{\mathrm{sn},\mathcal{I}}$ defined in \eqref{snisnr} with domain $H^2(\T_L)\times H^1(\T_L)$ has exactly two negative eigenvalue.
\end{cor}

\begin{proof}
In fact, let us start by considering $\vec{\mathcal{L}}_{\mathrm{sn},\mathcal{R}}$. We shall prove that there is an explicit relation between the negative eigenvalues of $\vec{\mathcal{L}}_{\mathrm{sn},\mathcal{R}}$ and the ones of $\mathcal{L}_{\mathrm{sn},\mathcal{R}}$. In fact, first of all notice that, in the same fashion as before, we deduce that $\vec{\mathcal{L}}_{\mathrm{sn},\mathcal{R}}$ has only point spectra. Hence, let us consider any negative eigenvalue $-\lambda^2\in\R$ with associated eigenfunction $(\phi_1,\phi_2)$, that is, $(\phi_1,\phi_2,\lambda)$ satisfies the equation \[
\vec{\mathcal{L}}_{\mathrm{sn},\mathcal{R}}(\phi_1,\phi_2)=-\lambda^2(\phi_1,\phi_2).
\]
Then, notice that the latter equation implies that $(1+\lambda^2)\phi_2=c\phi_1 $. Thus, by plugging this relation into the first component of the eigenvalue equation, we obtain \[
-\phi_{1,xx}-(1+c^2)\phi_{1,x}+3\big(\psi_c^{\mathrm{sn}}\big)^2\phi_1=-\lambda^2\Big(1+\tfrac{c^2}{1+\lambda^2}\Big)\phi_1=:-\mu^2\phi_1.
\]
Therefore, $-\mu^2\in\R$ is a negative eigenvalue of the operator $\mathcal{L}_{\mathrm{sn},\mathcal{R}}$. Now, we conclude the proof by following all the previous steps backwards. In fact, in the same fashion as before, if $-\mu^2\in\R$ is a negative eigenvalue of the operator $\mathcal{L}_{\mathrm{sn},\mathcal{R}}$ with associated eigenfunction $\phi_1$, then, by defining $\lambda^2$ given by the relation \[
\mu^2=\lambda^2\left(1+\tfrac{c^2}{1+\lambda^2}\right) \quad \,\hbox{and} \,\quad \phi_2:=\dfrac{c}{1+\lambda^2}\phi_1,
\]
we obtain that $-\lambda^2$ is a negative eigenvalue of the operator $\vec{\mathcal{L}}_{\mathrm{sn},\mathcal{R}}$ with corresponding eigenfunction $(\phi_1,\phi_2)$. Finally, it is worth to notice that this procedure finds all negative eigenvalues, what finish the proof for $\vec{\mathcal{L}}_{\mathrm{sn},\mathcal{R}}$. The spectral information about $\vec{\mathcal{L}}_{\mathrm{sn},\mathcal{I}}$ can be obtained in the same fashion.
\end{proof}

\subsection{Orbital Stability}

Finally, with all the above analysis we are able to prove the orbital stability result. In fact, we proceed in a similar fashion as in the previous section, however as we shall see, this case shall be slightly different. In fact, due to the presence of more than one unstable direction, we are not  exactly in the setting of \cite{GSS} but in the one of \cite{GSS2}. Hence, with all the above analysis we know that the stability/instability problem reduces to study the convexity/concavity of the scalar function: \[
d(c):=\big(\mathcal{E}-c\mathcal{F}\big)(\vec{\psi}_{c}^{\mathrm{sn}}).
\] 
We recall that in our current setting and by following the notation introduced at the beginning of this section, we have $\vec{\psi}_c^{\mathrm{sn}}=\big(\psi_{c}^{\mathrm{sn}},c\psi_{c}^{\mathrm{sn}},0,0\big)$. On the other hand, recalling that $\vec{\psi}_c^{\mathrm{sn}}$ is a critical point of the action functional, we deduce that \[
d'(c)=-c\int_0^L\big(\psi_c^{\mathrm{sn}}\big)^2dx.
\]
Since we still have to compute $d''(c)$, before going further it is convenient to express $d'(c)$ in terms of well-known functions. In fact, by using formula \eqref{sn_l2norm_formula}, identities in \eqref{recall_complex_snoidal} as well as relation \eqref{period_sn} we have
\[
c\int_0^L\big(\psi_c^{\mathrm{sn}}\big)^2dx=c\beta_2^2\int_0^L\mathrm{sn}^2(\ell x;\kappa)=\dfrac{32c}{L}\Big(K(\kappa)-E(\kappa)\Big)K(\kappa)
\]
Thus, by directly differentiating the latter identity with respect to $c$ we get: \begin{align}\label{derivative_complex_ddc}
\dfrac{L}{32}\dfrac{d}{dc}\left\{c\int_0^L\big(\psi_c^{\mathrm{sn}}\big)^2dx\right\}&=\Big(K(\kappa)-E(\kappa)\Big)K(\kappa)+c\Big(K'(\kappa)-E'(\kappa)\Big)\dfrac{d\kappa}{dc}K(\kappa)\nonumber
\\ & \quad +c\Big(K(\kappa)-E(\kappa)\Big)\dfrac{d\kappa}{dc}K'(\kappa)=:\mathrm{I}+\mathrm{II}+\mathrm{III}.
\end{align}
We shall prove that each of the previous terms $\mathrm{I}$, $\mathrm{II}$ and $\mathrm{III}$ is positive, what shall be enough to conclude the proof. Indeed, first of all notice that since we already know the monotonicity properties of $K$ and $E$, it is enough to study the positivity/negativity of $\tfrac{d\kappa}{dc}$. In fact, on the one-hand, by differentiating equation \eqref{period_sn} with respect to $c$ we obtain
\begin{align}\label{complex_dbdc}
\dfrac{d\beta_1}{dc}=\dfrac{4\sqrt{2}}{L}K'(\kappa)\dfrac{d\kappa}{dc}.
\end{align}
On the other hand, by differentiating the equation defining $\kappa$ in \eqref{recall_complex_snoidal} with respect to $c$ we get
\[
\dfrac{d\kappa}{dc}=-\dfrac{\sqrt{2\omega_c-\beta_1^2}}{\beta_1^2}\cdot \dfrac{d\beta_1}{dc}+\dfrac{2c}{\beta_1\sqrt{2\omega_c-\beta_1^2}}-\dfrac{1}{\sqrt{2\omega_c-\beta_1^2}}\cdot\dfrac{d\beta_1}{dc}.
\]
Gathering both identities and after some direct re-arrangements we obtain \[
\left(\dfrac{L}{4\sqrt{2}}+\dfrac{\sqrt{2\omega_c-\beta_1^2}}{\beta_1^2}K'(\kappa)+\dfrac{1}{\sqrt{2\omega_c-\beta_1^2}}K'(\kappa)\right)\dfrac{d\beta_1}{dc}=\dfrac{2c K'(\kappa)}{\beta_1\sqrt{2\omega_c-\beta_1^2}}.
\]
Therefore, by recalling that $K'>0$ together with the fact that $\beta_1\in(\sqrt{\omega_c},\sqrt{2\omega_c})$ we conclude that $\tfrac{d\beta_1}{dc}>0$, and hence, by using equation \eqref{complex_dbdc} again, we conclude that $\tfrac{d\kappa}{dc}>0$. 

\medskip

Now, we claim that this information is enough to conclude the sign of \eqref{derivative_complex_ddc}. In fact, first of all notice that, since $K(\kappa)>E(\kappa)$ for all $\kappa\in(0,1)$ it immediately follows that \[
\mathrm{I}=\Big(K(\kappa)-E(\kappa)\Big)K(\kappa)>0.
\]
In the same fashion, recalling that $K'(\kappa)>0$ for all $\kappa\in(0,1)$ and by using the fact that $\tfrac{d\kappa}{dc}$ also has a sign, we obtain \[
\mathrm{III}=c\Big(K(\kappa)-E(\kappa)\Big)\dfrac{d\kappa}{dc}K'(\kappa)>0.
\]
Then, recalling that from \eqref{Eprime} we have that $E'(\kappa)<0$, we deduce that \[
\mathrm{II}=c\Big(K'(\kappa)-E'(\kappa)\Big)\dfrac{d\kappa}{dc}K(\kappa)>0.
\]
Hence, gathering all the previous inequalities, going back to the function $d(c)$, we conclude that \[
d''(c)=-\dfrac{d}{dc}\left\{c\int_0^L\big(\psi_c^{\mathrm{sn}}\big)^2dx\right\}<0,
\]
for all $c\in(c_L,+\infty)$, where $c_L$ is defined in Proposition \ref{MT_SN_COMPLEX}. 

\medskip

Finally, we are in position to apply the stability result in \cite{GSS2}. In fact, first of all, we define the space $X_\mathrm{odd}^1$ as the space of functions belonging to $H^1(\T_L)\times L^2(\T_L)$  that are $(\mathrm{odd},\mathrm{odd})$ regarded as functions defined in the whole line. Then, by using Corollary \ref{final_spectral}, recalling that the first eigenfunction found in Proposition \ref{prop_sni_1} as well as the first two eigenfunctions found in Proposition \ref{op_sn_i_2} are the three even, we deduce that \[
n\left(\vec{\mathcal{L}}_\mathrm{sn}\big\vert_{X^1_{\mathrm{odd}}}\right)=0,\]
where $n(\cdot)$ stands for the number of negative eigenvalues of the operator. Moreover, we have that $\vec{\mathcal{L}}_\mathrm{sn}$ has exactly two zero-eigenvalues, more specifically, its kernel is spanned by \[
(\psi_{c,x}^{\mathrm{sn}},c\psi_{c,x}^{\mathrm{sn}},0,0) \quad \hbox{and}\quad (0,0,\psi_c^{\mathrm{sn}},c\psi_c^{\mathrm{sn}}),
\]
and the remaining part of the spectra is bounded away from zero. However, notice that only the last of these two zero-eigenfunctions belongs to $X^1_{\mathrm{odd}}$, and hence, the kernel of the operator $\vec{\mathcal{L}}_{\mathrm{sn}}$ restricted to $X^1_{\mathrm{odd}}$ is \textbf{simple}. On the other hand, we have just proved that $\mathfrak{p}(d''(c))=0$, where $\mathfrak{p}(\cdot)$ stands for the number of positive eigenvalues of $d''(c)$. Therefore we have \[
n\left(\vec{\mathcal{L}}_\mathrm{sn}\big\vert_{X^1_{\mathrm{odd}}}\right)=\mathfrak{p}(d''(c)),
\]
and hence, we conclude by applying the stability theorem in \cite{GSS2}. Specifically, we obtain the following result. 

\begin{thm}\label{stab_final}
Under the assumptions of Proposition \ref{MT_SN_COMPLEX}, the complex-valued snoidal wave solution given by \eqref{sn_complex_def} is orbitally stable in $X^1_{\mathrm{odd}}$ under the periodic flow of the $\phi^4$-equation.
\end{thm}

\medskip

\section{Appendix}

\subsection{Proof of Proposition \ref{MT_SN_CURVE}}\label{app_dn_sb_case}

In fact, first of all notice that for $c$ fixed, $T_{\mathrm{sb}}$ as in \eqref{sb_period} regarded as a function of $\beta_1$ satisfies $T_{\mathrm{sb}}((1,\sqrt{2}))=(2\pi\sqrt{\omega_{\mathrm{sb}}},+\infty)$. Moreover, notice that due to condition \eqref{csb_condition} we have the bound $2\pi\sqrt{\omega_{\mathrm{sb}}}<L$. Then, as an application of the Implicit Function Theorem, in order to conclude the uniqueness of $\beta_1(c)$ it is enough to show that $\tfrac{d}{d\beta_1}T_{\mathrm{sb}}<0$. Indeed, notice that by direct differentiation of the equation defining $\kappa$ in  \eqref{sn_sol_def} with respect to $\beta_1$ we have\begin{align*}
\dfrac{d\kappa}{d\beta_1}=-\dfrac{2}{\beta_1^2\sqrt{2-\beta_1^2}}<0 \quad \hbox{for all } \,\beta_1\in(1,\sqrt{2}). 
\end{align*} 
Therefore, recalling that $K(\kappa)$ is an strictly increasing function, and due to the sign of $\kappa'$ given by the latter inequality, by differentiating relation \eqref{sb_period} with respect to $\beta_1$ we conclude 
\begin{align}
\dfrac{d}{d\beta_1}T_{\mathrm{sb}}=-\dfrac{4\sqrt{2\omega_{sb}}}{\beta_1^2}K+\dfrac{4\sqrt{2\omega_\mathrm{sb}}}{\beta_1}K'(\kappa)\dfrac{d\kappa}{d\beta_1}<0,
\end{align} 
what finish the proof.\qed

\bigskip

\textbf{Acknowledgements :} The author is grateful to professor Fabio Natali for pointing out a flaw in the proof of Theorem \ref{stab_final} in a previous version of this work.

\end{document}